\documentclass[11pt,reqno]{amsart}

\usepackage[bookmarksdepth=3]{hyperref}
\usepackage[margin=1in]{geometry}
\usepackage{tikz-cd}
\usepackage{amsmath,graphicx,amsthm,amsfonts,amssymb,amsxtra,mathtools,dsfont}
\usepackage{enumitem}
\usepackage{color}
\usepackage{todonotes}
\usepackage{bm}

\newtheorem{theorem}{Theorem}[subsection]

\newtheorem{lemma}{Lemma}[subsection]
\makeatletter\let\c@lemma\c@theorem\makeatother

\makeatletter\let\c@alg\c@theorem\makeatother

\newtheorem{proposition}{Proposition}[subsection]
\makeatletter\let\c@proposition\c@theorem\makeatother

\makeatletter\let\c@conj\c@theorem\makeatother

\newtheorem{corollary}{Corollary}[subsection]
\makeatletter\let\c@corollary\c@theorem\makeatother

\theoremstyle{definition}
\newtheorem{definition}{Definition}[subsection]
\makeatletter\let\c@definition\c@theorem\makeatother

\makeatletter\let\c@convention\c@theorem\makeatother

\newtheorem{example}{Example}[subsection]
\makeatletter\let\c@example\c@theorem\makeatother

\newtheorem{remark}{Remark}[subsection]
\makeatletter\let\c@remark\c@theorem\makeatother

\makeatletter\let\c@fact\c@theorem\makeatother

\makeatletter\let\c@note\c@theorem\makeatother

\numberwithin{equation}{subsection}

\usepackage[capitalise]{cleveref}
\crefname{theorem}{Theorem}{Theorems}
\crefname{fact}{Fact}{Facts}
\crefname{note}{Note}{Notes}
\crefname{lemma}{Lemma}{Lemmas}
\crefname{alg}{Algorithm}{Algorithms}
\crefname{remark}{Remark}{Remarks}
\crefname{example}{Example}{Examples}
\crefname{proposition}{Proposition}{Propositions}
\crefname{conjecture}{Conjecture}{Conjectures}
\crefname{convention}{Convention}{Conventions}
\crefname{corollary}{Corollary}{Corollaries}
\crefname{definition}{Definition}{Definitions}
\crefname{equation}{\!\!}{\!\!} 

\DeclareMathOperator{\Hom}{Hom}
\DeclareMathOperator{\wt}{wt}

\newcommand{\Z}{\mathbb{Z}}
\newcommand{\Q}{\mathbb{Q}}
\newcommand{\N}{\mathbb{N}}

\renewcommand{\AA}{\mathcal{A}}
\newcommand{\A}{A}
\newcommand{\Orbit}{\mathfrak{O}}

\newcommand{\IntegerValuedRing}{\Z\langle T \rangle}

\newcommand{\gl}{\mathfrak{gl}}
\newcommand{\fh}{\mathfrak{h}}

\newcommand{\fb}{\mathfrak{b}}
\newcommand{\fl}{\mathfrak{l}}
\newcommand{\fp}{\mathfrak{p}}
\newcommand{\fu}{\mathfrak{u}}

\renewcommand{\k}{\mathds{k}}
\newcommand{\bi}{\mathbf{i}}
\newcommand{\bj}{\mathbf{j}}
\newcommand{\bk}{\mathbf{k}}

\newcommand{\bx}{\mathbf{x}}
\newcommand{\by}{\mathbf{y}}
\newcommand{\MM}{\mathcal{X}}

\newcommand{\End}{\operatorname{End}}
\newcommand{\Mat}{\operatorname{Mat}}
\newcommand{\MMat}{\mathcal{M}at}
\newcommand{\fS}{\mathfrak{S}}
\renewcommand{\SS}{\mathbb{S}}
\newcommand{\TT}{\mathbb{T}}

\newcommand{\NN}{\mathcal{N}}
\newcommand{\VV}{V}

\newcommand{\wideSinfinity}{\widehat{\mathfrak{S}}_{\infty}}
\newcommand{\Perm}{\operatorname{Perm}}

\newcommand{\Idealnd}{\mathcal{I}(n,d)}
\newcommand{\ev}{\operatorname{ev}}
\renewcommand{\dim}{\operatorname{dim}}
\newcommand{\Sntmod}{\SS(n,t)\text{-mod}}
\newcommand{\Sntmodlfd}{\SS(n,t)\text{-mod}_{\text{lfd}}}
\newcommand{\Sntmodfg}{\SS(n,t)\text{-mod}_{\text{fg}}}
\newcommand{\tr}{\operatorname{tr}}
\newcommand{\qrs}{q,r,s}
\newcommand{\dotU}{\dot{\mathbf{\mathcal{U}}}}
\newcommand{\balpha}{\bm{\alpha}}
\newcommand{\bbeta}{\bm{\beta}}
\newcommand{\blambda}{\bm{\lambda}}
\newcommand{\bmu}{\bm{\mu}}
\newcommand{\bnu}{\bm{\nu}}
\newcommand{\bgamma}{\bm{\gamma}}
\newcommand{\bdelta}{\bm{\delta}}
\newcommand{\CoDetBasis}{\mathcal{CB}}
\newcommand{\StdBasis}{\mathcal{SB}}
\newcommand{\CB}{\operatorname{CB}}

\allowdisplaybreaks

\title{Interpolating Schur Algebras}

\author{Addison Day}
\address{Department of Mathematics \\
		Oregon State University \\
		Corvallis, OR 97331, USA}
\email{dayadd@oregonstate.edu}

\author{Jonathan R.\  Kujawa}
\address{Department of Mathematics \\
		Oregon State University \\
		Corvallis, OR 97331, USA}
\email{kujawaj@oregonstate.edu}

\thanks{JRK was supported in part by Simons Foundation Grant SFI-MPS-TSM-00014272.}

\date{\today}


\begin{document}

\begin{abstract}
We introduce and study a one-parameter family of algebras that naturally generalize the Schur algebras.  We show the Schur algebra is canonically a quotient when the parameter is a nonnegative integer, characterize when they are semisimple, show they are based quasi-hereditary, and that their category of representations is a highest weight category that can be identified as a subcategory of parabolic category $\mathcal{O}$ for the general linear Lie algebra.
\end{abstract}

\maketitle
\tableofcontents

\section{Introduction}

\subsection{Overview}\label{SS:background}  Let $\k$ be a field of characteristic zero.  Let $V_{n}$ be an $n$-dimensional $\k$-vector space.  There is an action of the symmetric group on $d$ letters, $\fS_{d}$, on the $d$-fold tensor product $V_{n}^{\otimes d}$ by permutation of the tensor factors.  For more than a century, the Schur algebra 
\[
S(n,d) := \End_{\fS_{d}}\left(V_{n}^{\otimes d} \right)
\] has played a central role in representation theory.  Via Schur--Weyl duality it acts as a bridge between the representation theories of the symmetric and general linear groups.  It is also the prototypical example of a quasi-hereditary algebra and its representations are the prototypical example of a highest weight category.

The purpose of this paper is to introduce and study a family of algebras, $\SS(n,t)$, where $t \in \k$.  These algebras and their representations situate the classical Schur algebras within a one-parameter family in much the same way Deligne's category, $\operatorname{Rep}(\fS_{t})$, does for the representation theory of the symmetric group $\fS_{d}$ \cite{Deligne}.  Indeed, a key ingredient for this paper is the approach to $\operatorname{Rep}(\fS_{t})$ based on oligomorphic groups recently developed by Harman--Snowden~\cite{HarmanSnowden}.

An oligomorphic group is a group $G$ acting faithfully on a set $X$ in such a way that, for all $r \geq 0$, $G$ has finitely many orbits under the diagonal action on the $r$-fold product of $X$ with itself.  For example, $\fS_{\infty}$ acting on the natural numbers is an oligomorphic group.   Given an oligomorphic group, there is a notion of a measure and, hence, a theory of integration.  In the case of $\fS_{\infty}$, there is a unique measure $\mu_{t}$ for each $t \in \k$.  Given the data of an oligomorphic group with a measure, Harman--Snowden define a category of so-called permutation modules where the composition operation is given by convolution with respect to the measure. For a fixed $t \in \k$ we define $\SS(n,t)$ to be the path algebra of a subcategory of the corresponding permutation category for $\fS_{\infty}$:
\[
\SS(n,t) = \bigoplus_{\bbeta , \bdelta   \in \Lambda(n)_{t}} \Hom_{\Perm(\fS_{\infty};\mu_{t})}\left(\fS_{\infty}/\fS_{\bbeta },\fS_{\infty}/\fS_{\bdelta} \right).
\] See \cref{SS:Interpolating-Schur-Algebra} for a detailed definition.  The goal of the paper is to study the structure and representation theory of $\SS(n,t)$.

\subsection{Main Results}\label{SS:Main-Results}

Our first significant results are \cref{P:Snt-Basis-Theorem,T:product-formula}.  They give a double coset basis for $\SS(n,t)$ along with a combinatorial product rule that generalizes a known basis and product rule for $S(n,d)$.  As a consequence, $\SS(n,t)$ is seen to be the path algebra of a category first introduced by Harman in his thesis~\cite{HarmanThesis} and generalized in~\cite{Ryba}.  That the oligomorphic definition coincides with these earlier constructions will not surprise experts (e.g., the authors of~\cite{HarmanSnowden}).  These results provide a crucial connection to $S(n,d)$ that is used throughout the paper. 

The category $\Perm(\fS_{\infty};\mu_{t})$ is a spherical $\k$-linear tensor category.  We can therefore consider the semisimplification, $\overline{\Perm}(\fS_{\infty};\mu_{t})$, obtained by taking the quotient of $\Perm(\fS_{\infty};\mu_{t})$ by the tensor ideal of so-called negligible morphisms.  When $n \geq 2$, \cref{C:Semisimplification-and-finite-Snd} shows that the subcategory of $\Perm(\fS_{\infty};\mu_{t})$ that defines $\SS(n,t)$ is not semisimple if and only if $t \in \Z_{\geq 0}$.  In the case when $t = d \in \Z_{\geq 0}$, \cref{T:SSmodIdeal-is-isomorphic-to-Snd} demonstrates that the classical Schur algebra is the path algebra of the semisimplification:
\[
S (n,d) \cong \bigoplus_{\bbeta , \bdelta   \in \Lambda(n)_{d}} \Hom_{\overline{\Perm}(\fS_{\infty};\mu_{d})}\left(\fS_{\infty}/\fS_{\bbeta },\fS_{\infty}/\fS_{\bdelta} \right).
\]  In particular, there is a canonical quotient map $\SS(n,d) \twoheadrightarrow S(n,d)$.

In \cref{S:Codeterminant-Basis-for-SnT}  we establish that $\SS(n,t)$ has a basis that parallels Green's codeterminant basis for $S(n,d)$.  This gives $\SS(n,t)$ the structure of an upper finite based quasi-hereditary algebra in the sense of~\cite{BrundanStroppel} and makes the category of locally finite-dimensional $\SS(n,t)$-modules a highest weight category in the sense of \cite{CPS}.  In particular, there are standard modules $\Delta(\blambda )$ for $\blambda \in \Lambda^{+}(n)_{t}$, where $\Lambda^{+}(n)_{t}$ is the set of all ``partitions'' of $t$ with $n$ parts.  Each standard module has simple head $L(\blambda )$ and $\left\{L(\blambda ) \mid \blambda \in \Lambda^{+}(n)_{t} \right\}$ is a complete, irredundant set of simple $\SS(n,t)$-modules.  Moreover, \cref{T:Snt-when-t-is-natural-is-not-semisimple} classifies when $\SS(n,t)$ is semisimple: it is semisimple if and only if $n =1$ or if $ n\geq 2$ and $t \not\in \Z_{\geq 0}$.

Another consequence of the codeterminant basis is that $\SS(n,t)$ admits two natural integral forms.  There is the naive one given by the span of the double coset basis. The second and more interesting one is given by the span of the codeterminant basis.  This form should be viewed as the Schurification of $\SS(n,t)$ in the sense of~\cite{KleshchevMuth} and we expect this second form to have favorable properties (e.g., with respect to base change).  Also of interest, \cref{SS:naive-Z<T>-form,P:TnT-is-an-integral-form} demonstrate for $\k = \Q (T)$, the field of rational functions, that  $\SS (n,T)$ admits analogous forms over the ring of integer-valued polynomials.

Finally, in \cref{S:Enveloping-Algebra} we show that $\SS(n,t)$ can be realized as the quotient of an idempotented form of the enveloping algebra for the general linear Lie algebra, $\gl_{n}(\k )$.  This gives rise to an inflation functor from $\SS(n,t)$-modules to $U(\gl_{n}(\k ))$-modules which allows the category of finitely generated $\SS(n,t)$-modules to be identified as a full subcategory of an appropriate parabolic category $\mathcal{O}$ corresponding to the Levi subalgebra $\gl_{1}(\k ) \oplus \gl_{n-1}(\k )$.  A connection between Deligne's category and this parabolic category $\mathcal{O}$ was previously observed in~\cite{Aizenbud}.

\subsection{Future Directions}\label{SS:Future-Directions}  The results of this paper suggest a number of natural directions of investigation that we look forward to pursuing.  There should be an interpolation, $\SS_{q}(n,t)$, of the quantized Schur algebras $S_{q}(n,d)$. The oligomorphic group framework of~\cite{HarmanSnowden} can be used to define interpolations of various generalizations of the Schur algebra (e.g., the cyclotomic Schur algebras).  More exotic would be to formulate and study Delannoy Schur algebras~\cite{HarmanSnowdenSnyder} and Schur algebras for other interesting oligomorphic groups (e.g., see \cite[Section 1.3.7]{HarmanSnowden} for a list).  In a different direction, we anticipate that $\SS(n,t)$ can be used to study the stable (aka reduced) Kronecker coefficients.  A related question is to establish a Schur--Weyl duality relationship between $\SS (n,t)$ and Deligne's category.  Lastly, it is known that one can define $\operatorname{Rep}(GL_{r})$ for any $r \in \k$ and thereby meaningfully interpolate the dimension of $V_{n}$. There should be a corresponding two-variable interpolation of the Schur algebra, $\SS(r, t)$, where $r,t \in \k$.

\section{Notation and Combinatorics}\label{S:Notation-and-Combinatorics}
Let us establish notation, conventions, and combinatorics that will be used throughout the paper.

\subsection{Notation}\label{SS:Notation}  Let $\k$ denote a field of characteristic zero. Let $\N = \Z_{>0}$ denote the natural numbers and let $\Z_{\geq 0}$ denote the nonnegative integers.  Given $x,y \in \Z$ with $x \leq y$, let $[x,y] = \left\{x, x+1, \dotsc , y  \right\}$ and let $[x,\infty] = \left\{x, x+1, x+2, \dotsc  \right\}$. If $x>y$ we set $[x,y]=\emptyset$.  Given a set $X$, write $|X|$ for the cardinality of $X$.  For $k \geq 1$, let $X^{k}$ be the set of $k$-tuples of elements from $X$, and let $X^{(k)}$ be the set of all subsets of $X$ of size $k$.  If $X$ and $Y$ are $G$-sets for some group $G$, then $G$ acts diagonally on $X \times Y$.  If $X$ is a $G$-set for some group $G$, then $G$ acts on $X^{(k)}$ and on set partitions of $X$ by acting on the entries of each subset.  

Let $R$ be a commutative ring with $1$. Given $t \in R$ and $k \in \Z_{\geq 0}$, set 
\begin{equation}\label{E:binom-def}
\binom{t}{k} = \binom{t}{t-k,k} = \frac{t(t-1)(t-2)\dotsb  (t-k+1)}{k!}.
\end{equation}
By convention, $\binom{t}{0}=1$.  More generally, if $m, \ell \in \Z$ with $m \geq  \ell$, and $k_{2}, \dotsc , k_{r} \in \Z_{\geq 0}$, then set
\begin{equation}\label{E:multinom-def}
\binom{t-\ell}{t-m, k_{2}, \dotsc , k_{r}} = \frac{(t-\ell)(t-\ell-1)(t-\ell-2)\dotsb  (t-m+1)}{k_{2}!k_{3}!\dotsb k_{r}!}.
\end{equation}
We will generally assume that $R$ and $t$ are such that all expressions of the form \cref{E:binom-def,E:multinom-def} are elements of $R$. For example, this is automatic if $R$ contains $\mathbb{Q}$, or if $t$ is an integer and $m, \ell, k_{2}, \dotsc , k_{r}$ are chosen so that the above expressions are integer-valued polynomials in the variable $t$.

\subsection{Ring of Integer-Valued Polynomials}\label{SS:RingofIntegerValuedPolys}

Let 
\[
\IntegerValuedRing  = \left\{ p(T) \in \Q [T] \mid p(\Z )\subseteq \Z  \right\}
\] denote the ring of integer-valued polynomials in the variable $T$.  If $p(T) \in \Q [T]$ has degree $d$, then $p(T)$ is an element of $\IntegerValuedRing$ if and only if $p(a)$ is an integer for $d+1$ consecutive integers $a$.  It is also well known that $p(T) \in \Q [T]$ is an element of $\IntegerValuedRing$ if and only if 
\[
p(T) = \sum_{k=0}^{d} c_{k}\binom{T}{k}
\] for some $d \in \Z_{\geq 0}$ and $c_{0}, c_{1}, \dots, c_{d} \in \Z$.
 
Note that if $m, \ell, k_{2}, \dotsc , k_{r} \in \Z_{\geq 0}$ and $m = \ell + k_{2}+\dots + k_{r}$, then
\begin{equation}\label{E:multinomial-polynomial}
\binom{T-\ell}{T-m, k_{2}, \dotsc , k_{r}}
\end{equation}
is an integer-valued polynomial because it computes a multinomial coefficient when $T$ is evaluated at all sufficiently large integers.

\subsection{Classical Combinatorics}\label{SS:Combinatorics}

\subsubsection{Compositions and Partitions}\label{SS:Compositions-and-Partitions}  Given $n \in \N \cup \{\infty \}$ and $d \in \Z_{\geq 0}$, a \emph{composition} of $d$ with $n$ \emph{parts} is a sequence of nonnegative integers $\bgamma = (\gamma_{1}, \dotsc, \gamma_{n} )$ such that $|\bgamma| := \sum_{k \in [1,n]} \gamma_{k} = d$. Note that when $n=\infty$, this implies the sequence $\bgamma$ is eventually $0$.  Given a composition $\bgamma$, a \emph{part} of $\bgamma$ is one of the entries of the sequence. We emphasize that we allow parts to equal zero and those parts are not to be ignored.    A \emph{partition} of $d$ is a composition $\bgamma = (\gamma_{1}, \dots , \gamma_{n})$ of $d$ that satisfies $\gamma_{1} \geq  \dotsb \geq \gamma_{n}$.  For short, write $\bgamma \models d$ (resp., $\bgamma \vdash d$) when $\bgamma$ is a composition (resp., partition) of $d$.  There is a unique empty composition with zero parts and it is considered to be a partition of zero.

Given a composition $\bgamma = (\gamma_{1}, \dots , \gamma_{n})$ of $d$, there is an associated set partition $\left\{ \{\bgamma \}_{1}, \dots , \{\bgamma \}_{n} \right\}$ of $[1,d]$ given by
\begin{equation*}
\{\bgamma \}_{k} = \left[ 1+\sum_{i\in [1,k-1]}\gamma_{i},\sum_{i\in [1,k]} \gamma_{i} \right]
\end{equation*} for $k \in [1,n]$.
 In particular, $|\{\bgamma \}_{k}| = \gamma_{k}$ for all $k$.  Also, note that since we allow the parts of $\bgamma$ to equal zero, our definition of a set partition allows for the empty subset and it might appear multiple times.  

\begin{example}\label{Ex:Compositions-and-Partitions} The composition $\bgamma = (3,6,2,0,0)$ is a composition of $d=11$ with $n=5$ parts.  The set partition of $[1,11]$ corresponding to $\bgamma$ is 
\[
\left\{  \{\bgamma \}_{1} = \{1,2,3 \}, \{\bgamma \}_{2}=\{4,5,6,7,8,9 \}, \{\bgamma \}_{3}= \{10,11 \},  \{\bgamma \}_{4}=\emptyset,  \{\bgamma \}_{5}=\emptyset\right\}. 
\]  
\end{example}

Given $n \in \N \cup \left\{\infty \right\}$ and $d \in \Z_{\geq 0}$, write 
\begin{align*}
\Lambda(n,d) &= \left\{ \bgamma \models d \mid \text{$\bgamma$ has $n$ parts} \right\},\\
\Lambda^{+}(n,d) &= \left\{ \bgamma \vdash d \mid \text{$\bgamma$ has $n$ parts} \right\},\\
\Lambda(n) &= \bigcup_{d \in \Z_{\geq 0}} \Lambda(n,d),\\
\Lambda^{+}(n) &= \bigcup_{d \in \Z_{\geq 0}} \Lambda^{+}(n,d).
\end{align*}

For $n \in \N \cup \left\{\infty \right\}$, let 
\[
X(n) = \bigoplus_{i \in [1,n]} \Z\varepsilon_{i}
\] denote the free abelian group on the set $\left\{\varepsilon_{i} \right\}_{i \in [1,n]}$.  We identify $\Lambda(n)$ and its various subsets with their images in $X(n)$ under the embedding 
\[
\bgamma  \mapsto \sum_{r \in [1,n]} \gamma_{r}\varepsilon_{r}.
\]  As we did here, we sometimes leave it to the reader to infer that $\gamma_{r}$ denotes the $r$th part of $\bgamma$ (and likewise for $\blambda, \bbeta, \dotsc$).

\subsubsection{Matrices and Arrays}\label{SSS:Matrices}  Fix $n \in \N \cup \left\{\infty \right\}$.  Given compositions $\bbeta, \bdelta \in \Lambda(n,d)$, let $\Mat(n, d, \bdelta, \bbeta)$ denote the set of $n \times n$ matrices $q = (q_{i,j})_{i,j \in [1,n]}$ which satisfy the following conditions:
\begin{itemize}
\item for all $i,j \in [1,n]$, $q_{i,j} \in \Z_{\geq 0}$;
\item $\sum_{i,j \in [1,n]} q_{i,j}=d$;
\item for each $j \in [1,n]$, $\sum_{i \in [1,n]} q_{i,j} = \beta_{j}$;
\item for each $i \in [1,n]$, $\sum_{j \in [1,n]} q_{i,j} = \delta_{i}$.
\end{itemize}  Observe that $\Mat(n, d, \bdelta, \bbeta)$ is a finite set for any $\bbeta, \bdelta \in \Lambda(n,d)$.  Given $q = (q_{i,j}) \in \Mat(n,d, \bdelta, \bbeta)$, let $q^{*} = (q_{j,i}) \in \Mat(n, d, \bbeta, \bdelta)$ denote the transpose matrix. Set 
\[
\Mat(n,d) = \bigcup_{\bbeta, \bdelta \in \Lambda(n,d)} \Mat(n, d, \bdelta, \bbeta).
\]  Let $\Mat^{+}(n,d)$ and $\Mat^{+}(n, d, \bdelta, \bbeta )$ denote the subsets of upper triangular matrices.  Likewise,  let $\Mat^{-}(n,d)$ and $\Mat^{-}(n, d, \bdelta, \bbeta )$ denote the subsets of lower triangular matrices.

For any $\qrs \in \Mat(n,d)$, write $\A(n, \qrs)$ for the set of arrays $A = (A_{i,j,k})_{i,j,k \in [1,n]}$ which satisfy the following conditions:
\begin{itemize}
   \item for all $i, j, k \in [1,n]$, $A_{i,j,k} \in \Z_{\geq 0}$;
    \item for each pair $i,j \in [1,n]$, $\sum_{k \in [1,n]} A_{i,j,k} = q_{i,j}$;
    \item for each pair $j,k \in [1,n]$, $\sum_{i\in [1,n]} A_{i,j,k} = r_{j,k}$;
    \item for each pair $i, k \in [1,n]$, $\sum_{j\in [1,n]} A_{i,j,k} = s_{i,k}$.
\end{itemize}  For a fixed $\qrs \in \Mat(n,d)$, the set of arrays  $\A(n,\qrs)$ is finite.

\subsubsection{Semistandard Fillings of Young Diagrams and Semistandard Matrices}\label{SSS:Semi-Standard-Fillings}

Given a partition $\blambda = (\lambda_{1}, \dots , \lambda_{n})$ of $d$, there is the associated \emph{Young diagram of shape $\blambda$}, $Y^{\blambda}$. This is an array of one-by-one boxes with $\lambda_{i}$ boxes in the $i$th row, where the boxes are left justified.  A \emph{Young tableau of shape $\blambda$} is a filling of the boxes of $Y^{\blambda}$ with natural numbers. Call the filling \emph{row semistandard} if the entries are weakly increasing from left-to-right along rows, and call it \emph{semistandard} if it is  weakly increasing from left-to-right along rows and strictly increasing from top-to-bottom along columns.

Let $\blambda$ be a partition of $d$ and $\bbeta$ a composition of $d$.  Let $q= (q_{i,j}) \in \Mat(n, d, \blambda, \bbeta)$.  There is a unique tableau of shape $\blambda$ which is row semistandard and has $q_{i,j}$ entries equal to $j$ in row $i$ of $Y^{\blambda}$.  For short, we call this \emph{the filling of $Y^{\blambda}$ by $q$}.  Observe that a filling of $Y^{\blambda}$ cannot give a semistandard tableau if a $k$ appears anywhere in rows $k+1, \dotsc , n$ of $Y^{\blambda}$. That is, if the filling of $Y^{\blambda}$ by $q$ is semistandard, then $q$ must be upper triangular.

\subsubsection{Tuples and Weights}\label{SSS:Tuples-and-Weights}
Given $n \in \N \cup \left\{\infty \right\}$ and $d \in \N$, let 
\[
I(n,d) = \left\{\bi = (i_{1}, \dotsc , i_{d}) \mid i_{1}, \dotsc , i_{d} \in [1,n] \right\}
\] be the set of $d$-tuples with entries from $[1,n]$.  The symmetric group on $d$ letters, $\fS_{d}$, has a right action on $I(n,d)$ via place permutation. That is, the action is given by
\[
(i_1,\ldots,i_d)\cdot\sigma=(i_{\sigma (1)},\ldots,i_{\sigma (d)})
\]
for all $\sigma \in \fS_d$ and all $(i_1,\ldots,i_d)\in I(n,d).$

Given $\bi = (i_{1}, \dots , i_{d}) \in I(n,d)$, the \emph{weight} of $\bi$ is $\wt (\bi ) = \sum_{k \in [1,d]} \varepsilon_{i_{k}} \in X(n)$.  Observe that $\bi$ and $\bj$ are in the same $\fS_{d}$-orbit if and only if $\wt (\bi ) = \wt (\bj )$.  Let $\Orbit (n,d)$ denote the set of $\fS_{d}$-orbits in $I(n,d) \times I(n,d)$.  Write $\overline{(\bi, \bj )} \in \Orbit(n,d)$ for the orbit containing $(\bi , \bj ) \in I(n,d) \times I(n,d)$.

There is a well-defined function
\begin{align*}
\Orbit(n,d) &\to \Lambda(n,d) \times \Lambda(n,d), \\
\overline{(\bi, \bj )} &\mapsto (\wt (\bi), \wt (\bj )).
\end{align*}
  For $(\bmu , \bnu ) \in \Lambda(n,d) \times \Lambda(n,d)$, write $\Orbit(\bmu , \bnu)$ for the set of orbits in $\Orbit(n,d)$ which go to $(\bmu, \bnu)$ under this map.

\subsection{Combinatorics with a \texorpdfstring{$t$}{t}}\label{SS:t-Combinatorics}
\subsubsection{Compositions and Partitions with a \texorpdfstring{$t$}{t}}\label{SSS:compositions-and-partitions-of-t}
Let $n \in \N$.  Given a fixed $t \in R$, let
\[
\Lambda(n)_{t} = \left\{ \bbeta = \left( \beta_{1}, \beta_{2}, \dotsc , \beta_{n}\right) \in  \left(t- \Z  \right) \times \Z_{\geq 0}^{n-1} \; \left| \;  \sum_{i\in [1,n]} \beta_{i} = t \right. \right\},
\] and let
\[
\Lambda^{+}(n)_{t} = \left\{ \bbeta = \left( \beta_{1}, \beta_{2}, \dotsc , \beta_{n}\right) \in  \Lambda(n)_{t} \mid \beta_{2}\geq \dotsb \geq \beta_{n} \right\}.
\]  We extend these definitions to $n = \infty$ by assuming all but finitely many $\beta_{i}$ equal zero.

Given $\bgamma = (\gamma_{1}, \dotsc , \gamma_{n}) \in \Lambda(n)_{t}$, there is an associated set partition $\left\{ \{\bgamma \}_{1}, \dots , \{\bgamma \}_{n}  \right\}$ of $\Omega:=\N$ given by setting 
\[
\{\bgamma \}_{k} = \left[1+ \sum_{i \in [2,k-1]} \gamma_{i},\sum_{i \in [2,k]} \gamma_{i} \right]
\] for $k \in [2, n]$, and 
\[
\{\bgamma \}_{1} =  \Omega \backslash \bigcup_{k \in [2,n]} \{\bgamma \}_{k}.
\]
  Observe that $\{\bgamma \}_{1}$ is an infinite set and $|\{\bgamma \}_{k}| = \gamma_{k}$ when $k \in [2, n]$.  Also observe that we allow for parts of $\bgamma$ to equal zero and for the corresponding set partition of $\Omega$ to include empty subsets.

\begin{example}\label{Ex:t-compositions-and-partitions}  Consider the composition $\bgamma = (t-11,3,6,2,0,0) \in \Lambda(6)_{t} $.  The corresponding set partition of $\Omega$ is 
\[
\left\{  \{\bgamma \}_{1} = \{12, 13, 14, \dots  \}, \{\bgamma \}_{2}=\{1,2,3\},\{\bgamma \}_{3}=\{4,5,6,7,8,9\},\{\bgamma \}_{4}=\{10,11\}, \{\bgamma \}_{5} = \emptyset, \{\bgamma \}_{6}=\emptyset \right\}.
\] 
\end{example}

Given $n \in \N \cup \{\infty \}$ and $t \in R$, let   
\[
X(n)_{t} = \left\{\left. \blambda = \sum_{i \in [1,n]}\lambda_{i}\varepsilon_{i} \in \bigoplus_{i \in [1,n]} R\varepsilon_{i} \; \right| \;  \text{$\lambda_{1} \in t - \Z$, and $\lambda_{i} \in \Z $ for $i \in [2,n]$ }  \right\}.
\]
There is an embedding 
\begin{equation}\label{E:Lambda-embedded-in-X(n)}
\Lambda(n)_{t} \hookrightarrow X(n)_{t}
\end{equation}
given by $\bgamma = (\gamma_{1}, \dotsc , \gamma_{n}) \mapsto \sum_{i \in [1,n]} \gamma_{i}\varepsilon_{i}$.  We identify $\Lambda(n)_{t}$ and its subsets with their image under this map.  

\subsubsection{Matrices and Arrays with a \texorpdfstring{$t$}{t}}\label{SSS:matricies-of-t}

Given $\bbeta, \bdelta \in  \Lambda(n)_{t}$, let $\MMat_{n}(\bdelta, \bbeta)_{t}$ denote the set of matrices $q = (q_{i,j})_{i,j \in [1,n]}$ which satisfy the following conditions:
\begin{itemize}
\item $q_{1,1} \in t- \Z  $;
\item for all other $(i,j) \in [1,n]^{2}$, $q_{i,j} \in \Z_{\geq 0}$;
\item $\sum_{(i,j) \in [1,n]^{2}} q_{i,j}= t$;
\item for each $j \in [1,n]$, $\sum_{i \in [1,n]} q_{i,j} = \beta_{j}$;
\item for each $i \in [1,n]$, $\sum_{j \in [1,n]} q_{i,j} = \delta_{i}$.
\end{itemize} Observe that $\MMat_{n}(\bdelta, \bbeta)_{t}$ is a finite set for each pair $\bbeta, \bdelta \in \Lambda(n)_{t}$. Given $q = (q_{i,j}) \in \MMat_{n}(\bdelta, \bbeta)_{t}$, let $q^{*} = (q_{j,i}) \in \MMat_{n}(\bbeta, \bdelta)_{t}$ denote the transpose matrix. Set 
\[
\MMat_{n}(t) = \bigcup_{\bbeta, \bdelta \in \Lambda(n)_{t}}\MMat_{n}(\bdelta, \bbeta)_{t}.
\]  Let $\MMat_{n}^{+}(t)$ and $\MMat_{n}^{+}(\bdelta  , \bbeta )_{t}$ denote the subsets of upper triangular matrices.  Likewise, let $\MMat_{n}^{-}(t)$ and $\MMat_{n}^{-}(\bdelta  , \bbeta )_{t}$ denote the subsets of lower triangular matrices.

For $q, r, s \in \MMat_{n}(t)$, let $\AA_{n}(\qrs)_{t}$ be the set of arrays $A = (A_{i,j,k})_{i,j,k \in [1,n]}$ which satisfy:
\begin{itemize}
\item $A_{1,1,1} \in t - \Z $;
\item for all other $i, j, k \in [1,n]$, $A_{i,j,k} \in \Z_{\geq 0}$;
\item for each pair $i,j \in [1,n]$, $\sum_{k \in [1,n]} A_{i,j,k} = q_{i,j}$;
\item for each pair $j,k \in [1,n]$, $\sum_{i \in [1,n]} A_{i,j,k} = r_{j,k}$;
\item  for each pair $i, k \in [1,n]$, $\sum_{j \in [1,n]} A_{i,j,k} = s_{i,k}$.
\end{itemize}  For a fixed $\qrs \in \MMat_{n}(t)$, the set of arrays  $\AA_{n}(\qrs)_{t}$ is finite.

\begin{remark}\label{R:finite-inside-of-t-shifted}
Observe that $\Lambda(n,d)$ is naturally a subset of $\Lambda(n)_{d}$:
\[
\Lambda(n,d) = \left\{\bgamma = (\gamma_{1}, \gamma_{2}, \dotsc , \gamma_{n}) \in \Lambda(n)_{d} \mid \gamma_{1} \geq 0 \right\}.
\]  Likewise, for $\bgamma, \bdelta \in \Lambda(n,d)$, $\Mat(n,d, \bgamma, \bdelta )$ is the subset of $\MMat_{n}(\bgamma , \bdelta )_{d}$ consisting of matrices $q=(q_{i,j})$ such that $q_{1,1} \geq 0$. For $\qrs \in \Mat(n,d)$, $A(n, \qrs )$ is the subset of $\AA_{n}(\qrs )_{d}$ consisting of arrays $A = (A_{i,j,k})$ such that $A_{1,1,1} \geq 0$.
\end{remark}
\subsubsection{Linear Translation}\label{SSS:translation-of-t-shifted-combinatorics}
Given $t_{1}, t_{2}\in R$, it will be convenient to have notation for the operation of linearly translating from $t_{1}$-combinatorics to $t_{2}$-combinatorics.  For example, if $\bbeta = (\beta_{1}, \beta_{2},\dots , \beta_{n}) \in \Lambda(n)_{t_{1}}$, then we write $\bbeta(t_{2})$ for $(\beta_{1}+t_{2}-t_{1}, \beta_{2}, \dotsc , \beta_{n}) \in \Lambda(n)_{t_{2}}$.  Clearly, $\bbeta \mapsto \bbeta(t_{2})$ defines a bijection $\Lambda(n)_{t_{1}} \to \Lambda(n)_{t_{2}}$.  There are entirely similar bijections $\MMat_{n}(\bdelta, \bbeta)_{t_{1}} \to \MMat_{n}(\bdelta(t_{2}), \bbeta(t_{2}))_{t_{2}}$ given by $q \mapsto q(t_{2})$, $\AA_{n}(\qrs)_{t_{1}} \to \AA_{n}(q(t_{2}), r(t_{2}), s(t_{2}))_{t_{2}}$ given by $A \mapsto A(t_{2})$, etc.

\begin{remark}\label{R:Evaluation-at-large-d-is-a-bijection}
In the special case when $t_{2}=d$ is an integer and $\bbeta \in \Lambda(n)_{t_{1}}$, then $\bbeta(d)$ is an element of $\Lambda(n,d)$ whenever $d$ is large enough to ensure $\bbeta (d)_{1} \geq 0$.
Likewise, for $q \in \MMat_{n}(t_{1})$ and sufficiently large $d$, $q(d)$ can and will be viewed as an element of $\Mat(n, d)$;  given $A \in \AA_{n}(\qrs)_{t_{1}}$, $A(d)$ can and will be viewed as an element of $\A(n, q(d), r(d), s(d))$. 

Since for any $t \in \k $ and any $\bbeta, \bdelta \in \Lambda(n)_{t}$ the set $\MMat_{n}(\bdelta, \bbeta)_{t}$ is finite, the previous paragraph implies that the map 
\[
\MMat_{n}(\bdelta, \bbeta)_{t} \to  \Mat(n, d, \bdelta(d), \bbeta(d)) 
\] given by $q \mapsto q(d)$ will be a bijection for $d \gg  0$.  Likewise, since for any $q,r,s \in \MMat_{n}(t)$ the set $\AA_{n}(\qrs)_{t}$ is finite,  the map 
\[
\AA_{n}(\qrs)_{t} \to \A (n, q(d), r(d), s(d))  
\] given by $A \mapsto A(d)$ will be a bijection for $d \gg 0$.
\end{remark}

\subsubsection{Dominance Order}\label{SSS:Dominance-Order} 
Given $\bbeta = (\beta_{1}, \dotsc, \beta_{n}), \bdelta = (\delta_{1}, \dotsc , \delta_{n}) \in \Lambda(n)_{t}$, write $\bbeta \preceq  \bdelta$ if
\begin{equation}\label{E:dominance-order-definition}
\sum_{i \in [1,k]} \left(  \delta_{i} - \beta_{i} \right) \geq   0
\end{equation}
for all $k \in [1,n]$.  This defines a partial order on $\Lambda(n)_{t}$ which we call the \emph{dominance order}.  Observe that this partial order is preserved under linear translation from $\Lambda(n)_{t_{1}}$ to $\Lambda(n)_{t_{2}}$, and that it restricts to the usual dominance order on $\Lambda(n,d)$.

\subsubsection{Semistandard Shifted Matrices with a  \texorpdfstring{$t$}{t}}\label{SSS:Standard-T-shifted-matrices}

Given $\blambda \in \Lambda^{+}(n)_{t}$, $\bbeta \in \Lambda(n)_{t}$, and $q = (q_{i,j}) \in \MMat_{n}(\blambda, \bbeta)_{t}$, then we say $q$ is \emph{semistandard} if and only if the filling of $Y^{\blambda(d)}$ by $q(d)$ is semistandard for $d \gg 0$.

The following lemma is straightforward.

\begin{lemma}\label{L:right-codeterminant-are-standard}   
Given $\blambda \in \Lambda^{+}(n)_{t}$ and $\bbeta \in \Lambda(n)_{t}$, the matrix $q = (q_{i,j}) \in \MMat_{n}(\blambda, \bbeta)_{t}$ is semistandard if and only if, for all sufficiently large $d$, $q(d) = (q_{i,j}(d))$ satisfies
\begin{equation}\label{E:right-codeterminant-condition}
\sum_{j \in [1,y]}q_{x,j}(d) \leq  \sum_{j \in [1,y-1]} q_{x-1,j}(d) \text{ for all } (x,y) \in [2,n] \times [1,n].
\end{equation}
\end{lemma}

For example, if $q$ is semistandard, then $q_{2,1}$ is bounded by an empty sum and must equal zero.  Likewise, $q_{3,1}=0$. And $q_{3,2}=0$ because $q_{2,1}=0$.  Continuing in this way shows that $q$ will be upper triangular.  As another example, note that \cref{E:right-codeterminant-condition} imposes no conditions on $q_{1,y}$ when $y \in [2,n]$.  Also observe that \cref{E:right-codeterminant-condition} is automatically satisfied for $q_{2,y}$ when $y \in [2,n]$ because $d \gg 0$.

\section{The Classical Schur Algebra}\label{S:Classical-Schur-Algebra}
We recall various well known results about the Schur algebra.  These results will be used, both literally and inspirationally, in what follows.

\subsection{Double Cosets}\label{SSS:Double-Cosets}

Given a composition $\bgamma = (\gamma_{1}, \dots , \gamma_{n})$ of $d$, write $\fS_{\bgamma} \cong \fS_{\gamma_{1}}\times\dots \times \fS_{\gamma_{n}}$ for the corresponding Young subgroup of $\fS_{d}$. Given compositions $\bbeta, \bdelta \in \Lambda(n,d)$, there is a bijection  $\Mat(n, d, \bdelta, \bbeta) \to \fS_{\bdelta}\backslash \fS_{d}/ \fS_{\bbeta}$ given by 
\[
q \mapsto \fS_{\bdelta}\sigma \fS_{\bbeta},
\] where $\sigma \in \fS_{d}$ is any permutation with the property that, for all $i,j \in [1,n]$, 
\[
q_{i,j}= \left|  \{\bdelta \}_{i} \cap \sigma\left( \{\bbeta \}_{j}\right) \right|.
\]

Given $\bbeta, \bdelta \in \Lambda(n,d)$, there is also a bijection 
\[
\varphi:\Orbit(\bdelta, \bbeta) \to \Mat(n, d, \bdelta, \bbeta)
\]
given as follows.   For $\bk = (k_{1}, \dotsc , k_{d}) \in I(n,d)$ and $x \in [1,n]$, let $[\bk ]_{x} = \left\{r \in [1,d] \mid k_{r}= x \right\}$ be the set of subscripts where that entry of $\bk$ has value $x$.   Let $\overline{(\bj, \bi )} \in \Orbit(\bdelta, \bbeta)$.  Then the bijection is given by 
\begin{gather}\label{E:Orbit-to-Matrix-Bijection}
\varphi\left( \overline{(\bj, \bi )}\right) = \left(q_{x,y} \right)_{x,y \in [1,n]}, \\
                  q_{x,y} = \left| [\bj]_{x} \cap [\bi]_{y} \right|.
\end{gather}
That is, $q_{x,y}$ is a count of the number of positions which are simultaneously equal to $x$ in $\bj$ and equal to $y$ in $\bi$.

\subsection{The Classical Schur Algebra}\label{SS:Classical-Schur-Algebra}
Fix $n \in \N$ and $d \in \N $.  Let $V_{n}=\k^{n}$ be an $n$-dimensional $\k$-vector space with distinguished basis $\{v_{i} \}_{i\in [1,n]}$.  Then $V_{n}^{\otimes d}$ has distinguished basis 
\[
\left\{v_{\bi}=v_{i_{1}} \otimes \dotsb \otimes v_{i_{d}} \mid \bi = (i_{1}, \dotsc , i_{d}) \in I(n,d) \right\}.
\]
 There is a right $\fS_{d}$-action on $V_{n}^{\otimes d}$ induced by the place permutation action on $I(n,d)$.  The classical \emph{Schur algebra} is defined to be the endomorphism algebra 
\begin{equation*}
S(n,d)=S_{\k}(n,d) = \End_{\k \fS_{d}}\left(V_{n}^{\otimes d} \right).
\end{equation*}

For $\bi, \bj \in I(n,d)$ write $f_{\bj, \bi} : V_{n}^{\otimes d} \to V_{n}^{\otimes d}$ for the $\k$-linear map given by $f_{\bj , \bi} (v_{\bk}) = \delta_{\bi, \bk} v_{\bj}$ for all $\bk \in I(n,d)$.  Here $\delta_{\bi , \bk}$ denotes the Kronecker function.  Given $\overline{(\bj, \bi )} \in \Orbit(n,d)$, let $\xi_{\overline{(\bj, \bi )}} = \sum_{(\by, \bx ) \in \overline{(\bj, \bi )}} f_{\by, \bx}$.  We abuse notation and write $\xi_{\bj, \bi}$ for $\xi_{\overline{(\bj, \bi )}}$.

It is well known that 
\begin{equation}\label{E:Orbit-Sum-Basis}
\left\{\xi_{\bj, \bi} \mid \overline{(\bj, \bi )} \in \Orbit(n,d) \right\}
\end{equation}
forms a basis for $S(n,d)$. For example, see~\cite[Section 2.3]{Green}.  It will be convenient to use an alternate indexing of this basis using the function $\varphi$ from \cref{SSS:Double-Cosets} (cf.~\cite[Theorem 3.4]{DipperJames}).

Given $\bbeta \in \Lambda(n,d)$, let $V_{n, \bbeta}^{\otimes d}$ denote the span of $\left\{v_{\bi} \mid   \bi \in I(n,d), \wt (\bi ) = \bbeta  \right\}$. Then, 
\[
V_{n}^{\otimes d} = \bigoplus_{\bbeta \in \Lambda(n,d)} V_{n, \bbeta}^{\otimes d}
\] is a decomposition into $\k \fS_{d}$-submodules.  Furthermore,  $V_{n, \bbeta}^{\otimes d} \cong M^{\bbeta}$, where $M^{\bbeta}$ is the permutation module obtained by linearizing the $\fS_{d}$-set $X^{\bbeta}=\fS_{d}/\fS_{\bbeta}$.

Correspondingly, the Schur algebra has the direct sum decomposition 
\begin{equation}\label{E:Finite-Schur-Algebra}
S(n,d) \cong \bigoplus_{\bbeta, \bdelta \in \Lambda(n,d)} \Hom_{\k \fS_{d}} \left( M^{\bbeta}, M^{\bdelta} \right).
\end{equation}
Given $\bbeta, \bdelta \in \Lambda(n,d)$, observe that if $q \in \Mat(n, d, \bdelta, \bbeta)$, then $\xi_{\varphi^{-1}(q)} \in \Hom_{\k \fS_{d}}\left(M^{\bbeta}, M^{\bdelta} \right)$.  For short, write $\xi_{q}$ for $\xi_{\varphi^{-1}(q)}$.  From this and \cref{E:Orbit-Sum-Basis}, it follows that 
\begin{equation}\label{E:Double-Coset-Basis}
\left\{\xi_{q} \mid q \in \Mat(n, d, \bdelta, \bbeta) \right\}
\end{equation}
is a basis for $\Hom_{\k \fS_{d}} \left( M^{\bbeta}, M^{\bdelta} \right)$ and 
\[
\left\{\xi_{q} \mid q \in \Mat(n,d) \right\}
\] is a basis for $S(n,d)$.

\subsection{Classical Structure Constants}\label{SS:Classical-Structure-Constants}  Using this indexing the structure constants for the Schur algebra can be described combinatorially.   Fix $n \in \N$ and $d \in \Z_{\geq 0}$.  Let $q, r, s \in \Mat(n,d)$. Let $A(n, \qrs )$ be as in \cref{SSS:Matrices}.

For $A = (A_{i,j,k})_{i,j,k \in [1,n]} \in \A(n, \qrs)$, define
\begin{equation}\label{E:VA-def}
v_{A}=\prod_{i,k \in [1,n]} \binom{\sum_{j\in [1,n]} A_{i,j,k}}{A_{i,1,k}, \dotsc , A_{i,n,k}} = \prod_{i,k \in [1,n]} \frac{(\sum_{j\in [1,n]} A_{i,j,k})!}{\prod_{j \in [1,n]} (A_{i,j,k}!)}.
\end{equation}

For $\qrs \in \Mat(n,d)$, set
\begin{equation}\label{E:Classical-Structure-Constant-def}
c_{q,r}^s = \sum_{A \in \A(n, \qrs)} v_{A}.
\end{equation}

\begin{remark}
Observe that $\A(n, \qrs)$ is empty unless $q \in \Mat(n, d, \bgamma, \bdelta)$, $r \in \Mat(n, d, \bdelta, \bbeta)$, and $s \in \Mat(n, d, \bgamma, \bbeta)$ for some $\bbeta, \bdelta, \bgamma \in \Lambda(n,d)$. Set $c_{q,r}^s=0$ if $\A(n, \qrs)$ is empty.
\end{remark}

The structure constants for $S(n,d)$ are given by the following result found in~\cite[Proposition 2.3]{SantanaYudin} or~\cite[Proposition 6]{Ryba}. 

\begin{proposition}\label{T:classical-product} Fix $n \in \N$, $d \in \Z_{\geq 0}$, and $\bbeta, \bdelta, \bgamma \in \Lambda(n,d)$. For $q \in \Mat(n, d, \bgamma, \bdelta)$ and $r \in \Mat(n, d, \bdelta, \bbeta)$, the product in $S(n,d)$ is given by 
\begin{equation*}
\xi_{q} \xi_{r} = \sum_{s \in \Mat(n, d, \bgamma, \bbeta)} c_{q,r}^s \xi_s. 
\end{equation*}
\end{proposition}

\subsection{Green's Codeterminant basis for  \texorpdfstring{$S(n,d)$}{S(n,d)}}\label{SS:Finite-Codeterminant-Basis}

Green gave an alternate basis for $S(n,d)$ in~\cite{Green-codets} called the codeterminant basis, which we now describe.  Given $\blambda \in \Lambda^{+}(n,d)$ and $\bbeta, \bdelta  \in \Lambda(n,d)$, set 
\[
X(\blambda, \bbeta) = \left\{ \xi_{r} \mid  \text{$r \in \Mat(n, d, \blambda, \bbeta)$ and the filling of $Y^{\blambda}$ by $r$ is semistandard} \right\}
\] and
\[
Y(\bdelta, \blambda) = \left\{ \xi_{q} \mid  \text{$q \in \Mat(n, d, \bdelta, \blambda)$ and the filling of $Y^{\blambda}$ by $q^{*}$ is semistandard} \right\}. 
\]   The set 
\[
\CB (n,d, \bdelta, \bbeta ):=\left\{\xi_{q}\xi_{r} \left| (\xi_{q}, \xi_{r}) \in \bigcup_{\blambda \in \Lambda^{+}(n,d)} Y(\bdelta, \blambda) \times X(\blambda, \bbeta) \right. \right\}
\] is Green's codeterminant basis for $\Hom_{\k \fS_{d}} \left( M^{\bbeta}, M^{\bdelta} \right)$. The set 
\[
\CB(n,d):=\bigcup_{\bbeta,\bdelta\in\Lambda(n,d)}\CB(n, d, \bdelta,\bbeta)
\] is Green's codeterminant basis for $S(n,d)$.  This basis makes the classical Schur algebra a based quasi-hereditary algebra in the sense of~\cite{KleshchevMuth}.

\begin{remark}\label{R:Classical-Schur-Algebra-when-n-is-infinity}  The results of this section remain true when $n=\infty$ if one takes, as we do, the right hand side of \cref{E:Finite-Schur-Algebra} as the definition of $S(\infty, d)$.  
\end{remark}

\section{The Interpolating Schur Algebra}\label{S:Interpolating-Schur-Algebra}
We now define the main object of study in this paper.

\subsection{The Infinite Symmetric Group and Double Cosets}\label{SS:Infinite-Symmetric-Group}
Let $\fS_{\infty}$ be the group of bijections from $\Omega :=\N$ to itself.

Given $\bbeta = (\beta_{1}, \beta_{2}, \dotsc , \beta_{n}) \in \Lambda(n)_{t}$ there is an associated Young subgroup $\fS_{\bbeta}  \leq \fS_{\infty}$ isomorphic to
\[
 \fS_{\beta_{2}} \times \dots \times \fS_{\beta_{n}} \times \fS_{\infty}.
\] There is also the corresponding $\fS_{\infty}$-set of left cosets, $\MM^{\bbeta}:=\fS_{\infty} /\fS_{\bbeta}$.

Given $\bbeta, \bdelta \in \Lambda(n)_{t}$, the set of double cosets $\fS_{\bdelta}\backslash \fS_{\infty}/ \fS_{\bbeta}$ is in bijection with the set $\MMat_{n}(\bdelta, \bbeta)_{t}$ via the map 
\begin{equation}\label{E:infinite-q-to-matrix-bijection}
q \mapsto \fS_{\bdelta}\sigma \fS_{\bbeta},
\end{equation}
where $\sigma \in \fS_{\infty}$ is any permutation with the property that, for all $i,j \in [1,n]$ except $(i,j)=(1,1)$, 
\[
q_{i,j} = \left|\{ \bdelta \}_{i} \cap \sigma\left( \{\bbeta \}_{j}\right)\right|.
\]

Note that the set of $\fS_{\infty}$-orbits in $\fS_{\infty}/\fS_{\bdelta}\times \fS_{\infty}/\fS_{\bbeta}$ is in bijection with the set of double cosets $\fS_{\bdelta}\backslash \fS_{\infty}/\fS_{\bbeta}$. The map is given on orbit representatives by  
\begin{equation}\label{E:Orb-Lemma}
(x\fS_{\bdelta},y\fS_{\bbeta}) \mapsto \fS_{\bdelta}x^{-1}y\fS_{\bbeta}.
\end{equation}  We will freely identify these sets via this bijection.  Note that we can and will assume orbit representatives have the form $(\fS_{\bdelta},z\fS_{\bbeta})$ for some $z \in \fS_{\infty}$.  The bijection then becomes $(\fS_{\bdelta},z\fS_{\bbeta}) \mapsto \fS_{\bdelta}z\fS_{\bbeta}$.

\subsection{ \texorpdfstring{$\fS_{\infty}$}{Sinfty} as an Oligomorphic Group}\label{SS:Sinfty-as-a-Oligomorphic-Group}

 The action of $\fS_{\infty}$ on $\Omega$ is faithful and the diagonal action of $\fS_{\infty}$ on $\Omega^{r}$ has finitely many orbits for all $r \geq 1$.  This makes the pair $(\fS_{\infty}, \Omega)$ an oligomorphic group.  We refer the reader to~\cite{HarmanSnowden,Cameron} for details.  Here we only recall enough to make sense of the concepts we use.

There is a topology on $\fS_{\infty}$ given by declaring that the subgroups of $\fS_{\infty}$ which fix a finite subset of $\Omega$ are a base of open neighborhoods of the identity.  The action of $\fS_{\infty}$ on a set $X$ is \emph{smooth} if every point stabilizer is an open subgroup of $\fS_{\infty}$.  In this paper, all $\fS_{\infty}$-sets are assumed to be smooth.  An $\fS_{\infty}$-set is called \emph{finitary} if the action of $\fS_{\infty}$ has finitely many orbits.  Note that if $X$ and $Y$ are finitary $\fS_{\infty}$-sets, then $X \times Y$ is a finitary $\fS_{\infty}$-set under the diagonal action (see~\cite[Proposition 2.8(b)]{HarmanSnowden}).

More generally, if there exists an open subgroup $U \subseteq \fS_{\infty}$ acting smoothly on some set $X$ (which may or may not be a $\fS_{\infty}$-set), then we call $X$ a $\wideSinfinity$-set. If the action of $U$ has finitely many orbits, then we say it is a finitary $\wideSinfinity$-set.  Being a finitary $\wideSinfinity$-set does not depend in an important way on the choice of open subgroup $U$. See~\cite[Section I.2]{HarmanSnowden}.  Note that, because Young subgroups are open, $\fS_{\infty}/\fS_{\bbeta}$ is a finitary $\fS_{\infty}$-set for every $\bbeta \in \Lambda(n)_{t}$.  Write $\mathcal{M}^{\bbeta}$ for the linearization of $\MM^{\bbeta}=\fS_{\infty}/\fS_{\bbeta}$.

Recall that $\k$ denotes a characteristic zero field.  Harman--Snowden define the notion of a \emph{$\k$-valued measure} associated to an oligomorphic group such as $\fS_{\infty}$. This is a $\k$-valued function on the collection of finitary $\wideSinfinity$-sets subject to various axioms.  See~\cite[Section I.3]{HarmanSnowden} for specifics.  Important for this paper is the fact that for each $t \in \k $ there exists a unique measure $\mu_t$ satisfying $\mu_t(\Omega) = t$.  A description of $\mu_{t}$ is given as follows.

\begin{example}\label{Ex:Measure-Computation} For $p \in \Z_{\geq 0}$, let $\fS (p)$ be the subgroup of $\fS_{\infty}$ which fixes $[1,p]$ pointwise.  As discussed in~\cite[Example 3.4]{HarmanSnowden}, for a finitary $\wideSinfinity$-set $X$ there is a polynomial $p_{X}(T) \in \Q [T]$ such that 
\[
p_{X}(p) = \left| X^{\fS (p)}  \right|
\] for all sufficiently large $p$.  The measure $\mu_{t}$ satisfies $\mu_{t}(X) = p_{X}(t)$ for all $t \in \k$.

With this description it is straightforward to see that
\[
\mu_{t}\left( \fS_{\infty}/\fS_{\bbeta} \right) = \binom{t}{\beta_{1}, \beta_{2}, \dotsc , \beta_{n}}
\]  for all $\bbeta = (\beta_{1}, \dotsc , \beta_{n}) \in \Lambda(n)_{t}$.
\end{example}

\subsection{The Category  \texorpdfstring{$\Perm (\fS_{\infty};\mu_{t})$}{Perm(Sinfty; mut)}}\label{S:Perm-and-SPerm}

Let $X$ be a finitary $\wideSinfinity$-set, where $U$ is an open subgroup of $\fS_{\infty}$ which acts smoothly on $X$ with finitely many orbits. If $f: X \to \k$ is a function that is constant on $U$-orbits, then it is called a \emph{Schwartz function on $X$}.  Note that  Schwartz functions are defined for non-finitary sets in \cite{HarmanSnowden}, but that generality is not needed for this paper. If $f$ is a Schwartz function on $X$ and $c_{1}, \dotsc , c_{z} \in \k$ are the distinct values of $f$, then the integral of $f$ on $X$ with respect to the measure $\mu_{t}$ is defined to be
\begin{equation*}
\int_X f(x) \, dx = \sum_{i=1}^{z} c_i \mu_{t}(f^{-1}(c_i)).
\end{equation*}

Given finitary $\wideSinfinity$-sets $X$, $Y$, and $Z$, and Schwartz functions $f: Y \times X \to \k $ and $g: Z \times Y \to \k $, then $g \circ f : Z \times X \to \k  $ is the Schwartz function given by their convolution:
\begin{equation}\label{E:Composition-of-Schwartz-Functions}
    (g \circ f)(z, x) = \int_Y g(z, y) f(y, x) \, dy.
\end{equation}

The category $\Perm(\fS_{\infty}; \mu_t)$ is the category of all finitary $\fS_{\infty}$-sets with $\Hom_{\Perm(\fS_{\infty}; \mu_t)}(X, Y)$ being the $\k$-vector space of Schwartz functions on the finitary $\fS_{\infty}$-set $Y \times X$. Composition of morphisms in $\Perm(\fS_{\infty}; \mu_t)$ is the composition of Schwartz functions given in \cref{E:Composition-of-Schwartz-Functions}.

\subsection{The Interpolating Schur Algebra}\label{SS:Interpolating-Schur-Algebra}

\begin{definition}\label{D:Interpolating-Schur-Algebra} Fix $t \in \k $ and $n \in \N \cup \left\{\infty \right\}$. Let $\mu_t$ be the $\k$-valued measure on finitary $\wideSinfinity$-sets which is described in \cref{Ex:Measure-Computation}.  The Schur algebra associated to this data is the path algebra for the full subcategory of $\Perm(\fS_{\infty} ; \mu_{t})$ consisting of objects $\MM^{\bbeta}$ for $\bbeta \in \Lambda(n)_{t}$:
\begin{equation}\label{E:Interpolating-Schur-Algebra-Definition}
\SS_{\k }(n,t) = \bigoplus_{\bbeta, \bdelta \in \Lambda(n)_{t}} \Hom_{\Perm(\fS_{\infty}; \mu_{t})}\left( \MM^{\bbeta}, \MM^{\bdelta}\right).
\end{equation}
\end{definition} \noindent  We frequently leave the field of definition implicit and write $\SS(n,t)$ for $\SS_{\k }(n,t)$.

There is an obvious basis for $\SS(n,t)$.  Namely, given $\bbeta, \bdelta \in \Lambda(n)_{t}$, consider $\fS_{\infty}/\fS_{\bdelta} \times \fS_{\infty}/\fS_{\bbeta}$.  Since this is a finitary $\fS_{\infty}$-set, the space of Schwartz functions has a $\k$-basis given by the indicator functions for the orbits.  By \cref{SS:Infinite-Symmetric-Group} these are indexed by $\MMat_{n}(\bdelta, \bbeta)_{t}$.  Let 
\[
\xi_{q} \in \Hom_{\Perm(\fS_{\infty}; \mu_t)}\left(\MM^{\bbeta}, \MM^{\bdelta} \right)
\]
denote the indicator function which is $1$ on the orbit labelled by $q \in \MMat_{n}(\bdelta, \bbeta)_{t}$ and $0$ on all other orbits.  Then the following result holds.  

\begin{proposition}\label{P:Snt-Basis-Theorem}  Fix $t \in \k $ and $n \in \N \cup \left\{\infty \right\}$. For each $\bbeta, \bdelta \in \Lambda(n)_{t}$,
\[
\left\{\xi_{q} \mid q \in \MMat_{n}(\bdelta, \bbeta)_{t} \right\}
\] is a $\k$-basis for $\Hom_{\Perm(\fS_{\infty};  \mu_{t})}\left( \MM^{\bbeta}, \MM^{\bdelta}\right)$ and each such summand given in \cref{E:Interpolating-Schur-Algebra-Definition}  is finite-dimensional as a $\k$-vector space.

Moreover, the set 
\[
\StdBasis(n)_{t} :=\left\{\xi_{q} \mid q \in \MMat_{n}(t) \right\}
\] is a $\k$-basis for $\SS(n,t)$.  
\end{proposition}

\subsection{Structure Constants}\label{SS:Structure-Constants-for-Snt}

We next derive a combinatorial description of the structure constants for $\SS(n,t)$.

Throughout this section let $q = (q_{i,j}) \in \MMat_{n}(\bgamma, \bdelta)_{t}$ and $r = (r_{i,j}) \in \MMat_{n}(\bdelta, \bbeta)_{t}$. Consider
\begin{equation}\label{E:Snt-structure-constants}
    \xi_{q} \circ \xi_{r} = \sum_{s \in \MMat_{n}(\bgamma, \bbeta)_{t}} C_{q,r}^s(t) \xi_s,
\end{equation} where $C_{q,r}^{s}(t) \in \k$.

For $s = (s_{i,j}) \in \MMat_{n}(\bgamma, \bbeta)_{t}$, $\xi_s$ is the indicator function for some double coset, say $\fS_{\bgamma}\sigma\fS_{\bbeta}$.   It follows that
\begin{equation}\label{E:C-as-Integral}
C_{q,r}^s (t) = (\xi_{q} \circ \xi_{r})(\fS_{\bgamma} \sigma\fS_{\bbeta}) = \int_{\fS_{\infty}/\fS_{\bdelta}} \xi_{q}(\fS_{\bgamma}, \theta\fS_{\bdelta}) \xi_{r}(\theta\fS_{\bdelta}, \sigma\fS_{\bbeta}) \, d\left(\theta\fS_{\bdelta} \right).
\end{equation}
However, the function
\[
\xi_{q}(\fS_{\bgamma},-)\xi_{r}(-,\sigma\fS_{\bbeta}): \fS_{\infty}/\fS_{\bdelta} \to \k
\] only takes on values $0$ and $1$. Therefore, the integral in \cref{E:C-as-Integral} is simply the measure of the set of cosets $\theta\fS_{\bdelta}$ such that the double coset $\fS_{\bgamma}\theta\fS_{\bdelta}$ has indicator function  $\xi_{q}$ and, simultaneously,  the double coset $\fS_{\bdelta}\theta^{-1}\sigma\fS_{\bbeta}$ has indicator function  $\xi_{r}$.
We let $X(\qrs)$ denote this set of cosets.  Then, 
\begin{equation}\label{E:X(qrs)}
C_{q, r}^s(t)=\mu_t(X(\qrs)).
\end{equation}
To compute the measure of $X(\qrs)$ requires some preparatory lemmas.

It is straightforward to verify the following result.

\begin{lemma}\label{L:U-action}
The subgroup $U(\bgamma, \sigma, \bbeta) = \fS_{\bgamma} \cap \sigma\fS_{\bbeta}\sigma^{-1}$ acts on $X(\qrs)$ by left translation. The stabilizer of $\theta\fS_{\bdelta} \in X(\qrs)$ is equal to $V(\bgamma, \sigma, \bbeta, \theta, \bdelta)=  U(\bgamma, \sigma, \bbeta) \cap \theta\fS_{\bdelta}\theta^{-1}$.
\end{lemma}

\begin{lemma}\label{L:orbit-bijection}
The $U(\bgamma, \sigma, \bbeta)$-orbits on $X(\qrs)$ are in bijection with $\AA_{n}(\qrs)_{t}$.
\end{lemma}
\begin{proof}
  For short, write $U=U(\bgamma, \sigma, \bbeta)$. Define a map
    \begin{align*}
        X(\qrs)&\to\AA_{n}(\qrs)_{t}\\
        \theta\fS_{\bdelta}&\mapsto (A_{i,j,k})_{i,j,k \in [1,n]},
    \end{align*}
    where 
    
\begin{equation*}
A_{i,j,k}=\begin{cases}
\left| \{\bgamma\}_i\cap\theta\{\bdelta\}_j\cap\sigma\{\bbeta\}_k \right|, & \text{if }(i,j,k)\neq(1,1,1);\\
    t-\sum_{(i,j,k) \in [1,n]^{3} \backslash \{(1,1,1) \}}A_{i,j,k}, & \text{for }(i,j,k)=(1,1,1).
\end{cases}
\end{equation*}
That $A=(A_{i,j,k})$ lies in $\AA_{n}(\qrs)_{t}$ is a straightforward verification. We next claim the function is constant on $U$-orbits. Say $u\in U$ satisfies $u\theta\fS_{\bdelta}=\pi\fS_{\bdelta}.$ Then $u\theta =\pi x$ for some $x\in\fS_{\bdelta}$. Using that $u$ is an element of $U=\fS_{\bgamma}\cap\sigma\fS_{\bbeta}\sigma^{-1}$ and applying $u$ to the set $\{\bgamma\}_i\cap\theta\{\bdelta\}_j\cap\sigma\{\bbeta\}_k$ yields
\[
\left| \{\bgamma\}_i\cap\theta\{\bdelta\}_j\cap\sigma\{\bbeta\}_k \right|=\left| \{\bgamma\}_i\cap \pi\{\bdelta\}_j\cap \sigma\{\bbeta\}_k \right|
 \]
        for all $(i,j,k)\neq(1,1,1)$. Thus the array defined using $\theta\fS_{\bdelta}$ is the same as the array defined using $\pi\fS_{\bdelta}$.  Now consider the induced function from the set of $U$-orbits to $\AA_{n}(\qrs)_{t}$.
        
We first show the map is injective. Suppose $\theta\fS_{\bdelta}$ and $\pi\fS_{\bdelta}$ map to the same array.  Then, for all $i,j,k \in [1,n]$, 
\begin{equation*}
\left|  \{\bgamma\}_i\cap\pi\{\bdelta\}_j\cap\sigma\{\bbeta\}_k\right|=\left| \{\bgamma\}_i\cap\theta\{\bdelta\}_j\cap\sigma\{\bbeta\}_k\right|.
\end{equation*} The collections
\begin{equation*}
   \left\{ \{\bgamma\}_i\cap\pi\{\bdelta\}_j\cap\sigma\{\bbeta\}_k \mid (i,j,k) \in [1,n]^{3}  \right\}
\end{equation*} and
\begin{equation*}
  \left\{ \{\bgamma\}_i\cap\theta\{\bdelta\}_j\cap\sigma\{\bbeta\}_k \mid (i,j,k) \in [1,n]^{3} \right\}
 \end{equation*}  are two set partitions of $\Omega$. By assumption, the subsets associated to $(i,j,k)$ in each set partition have the same cardinality. Therefore there is an element $u\in \fS_{\infty}$ which takes each subset in the first set partition to its mate in the second set partition.  That is, 
\begin{equation}\label{sets}
 u(\{\bgamma\}_i\cap\pi\{\bdelta\}_j\cap\sigma\{\bbeta\}_k)=\{\bgamma\}_i\cap\theta\{\bdelta\}_j\cap\sigma\{\bbeta\}_k
 \end{equation}
 for all $i,j,k \in [1,n]$.

We claim $u\in U=\fS_{\bgamma}\cap\sigma\fS_{\bbeta}\sigma^{-1}$. First note that, for any $i \in [1,n]$,
 \begin{align*}
     u\{\bgamma\}_i&=u\bigcup_{j,k \in [1,n]}\{\bgamma\}_i\cap\pi\{\bdelta\}_j\cap\sigma\{\bbeta\}_k\\
     &=\bigcup_{j,k \in [1,n]}u(\{\bgamma\}_i\cap\pi\{\bdelta\}_j\cap\sigma\{\bbeta\}_k)\\
     &=\bigcup_{j,k \in [1,n]}\{\bgamma\}_i\cap\theta\{\bdelta\}_j\cap\sigma\{\bbeta\}_k,\\
     &=\{\bgamma\}_i.
 \end{align*}
 Hence $u\in\fS_{\bgamma}$. On the other hand, for any $k \in [1,n]$,
\[
\sigma\{\bbeta\}_k= \bigcup_{i,j \in [1,n]}\{\bgamma\}_i\cap \pi\{\bdelta\}_j\cap \sigma\{\bbeta\}_k.
\]
Applying $u$ to both sides yields
\begin{align*}
    u\sigma\{\bbeta\}_k&=u\bigcup_{i,j \in [1,n]}\{\bgamma\}_i\cap\pi\{\bdelta\}_j\cap\sigma\{\bbeta\}_k\\
    &=\bigcup_{i,j \in [1,n]}u(\{\bgamma\}_{i}\cap\pi\{\bdelta\}_j\cap\sigma\{\bbeta\}_k)\\
    &=\bigcup_{i,j \in [1,n]}\{\bgamma\}_{i}\cap\theta\{\bdelta\}_j\cap\sigma\{\bbeta\}_k\\
     &=\sigma\{\bbeta\}_k.
\end{align*}
This shows $u\sigma\{\bbeta\}_k=\sigma\{\bbeta\}_k$ and, hence, $u\in\sigma\fS_{\bbeta}\sigma^{-1}$. So $u\in U=\fS_{\bgamma}\cap\sigma\fS_{\bbeta}\sigma^{-1}$, as claimed.

We next show that $u\pi\fS_{\bdelta}=\theta\fS_{\bdelta}$. Since $u \in U$ satisfies \cref{sets}, it follows that for all $j \in [1,n]$,
\begin{align*}
    u\pi\{\bdelta\}_j&=\bigcup_{i,k\in [1,n]}\{\bgamma\}_i\cap u\pi\{\bdelta\}_j\cap\sigma\{\bbeta\}_k\\
    &=\bigcup_{i,k\in [1,n]}u\{\bgamma\}_i\cap u\pi\{\bdelta\}_j\cap u\sigma\{\bbeta\}_k,\\
    & = \bigcup_{i,k\in [1,n]}u(\{\bgamma\}_i\cap \pi\{\bdelta\}_j\cap\sigma\{\bbeta\}_k)\\
    &=\bigcup_{i,k\in [1,n]}\{\bgamma\}_i\cap \theta\{\bdelta\}_j\cap\sigma\{\bbeta\}_k,\\
    &=\theta\{\bdelta\}_j.
\end{align*}
  But $\fS_{\bdelta}$ is the stabilizer of the set partition $\{ \{\bdelta \}_{1}, \dotsc , \{\bdelta \}_{n} \}$ and so we conclude $\theta^{-1}u\pi \in \fS_{\bdelta}$.  Hence, $u\pi\fS_{\bdelta}=\theta\fS_{\bdelta}$, as claimed.

Finally, we show the map is surjective. Let $A=(A_{i,j,k})\in\AA_{n}(\qrs)_{t}$. For any $i,k \in [1,n]$, set $X_{i,k}=\{\bgamma\}_i\cap\sigma\{\bbeta\}_k$.  On the one hand, $\fS_{\bgamma }\sigma\fS_{\bbeta }$ is the double coset corresponding to the indicator function $\xi_{s}$.  Applying \cref{E:infinite-q-to-matrix-bijection}, this shows that the cardinality of $X_{i,k}$ equals $s_{i,k}$ when $(i,k) \neq (1,1)$ and the cardinality of $X_{1,1}$ is infinite.  On the other hand,  $\sum_{j \in [1,n]} A_{i,j,k} = s_{i,k}$ for all $i,k \in [1,n]$. Taken together, this implies that for each pair $(i,k)$, there exists a set partition of $X_{i,k}$, $\left\{X_{i,j,k} \mid j \in [1,n] \right\}$, such that $X_{i,j,k}$ has cardinality $A_{i,j,k}$ for $(i,j,k)\neq(1,1,1)$ and $X_{1,1,1}$ is infinite. Observe that  $\{X_{i,j,k} \mid i,j,k \in [1,n]\}$ provides a set partition of $\Omega$.  Also observe that for $j \in [2,n]$,
\[
|\{\bdelta\}_j|= \sum_{i,k \in [1,n]} A_{i,j,k}=\left|  \bigcup_{i,k \in [1,n]}X_{i,j,k}\right|,
\]  and both $\{\bdelta \}_{1}$ and $\bigcup_{i,k \in [1,n]} X_{i,1,k}$ are infinite.  Choose $\theta\in \fS_{\infty}$ with the property that it maps  $\{\bdelta\}_j$ bijectively to $\bigcup_{i,k}X_{i,j,k}$ for all $j \in [1,n]$. By construction, $|\{\bgamma\}_i\cap\theta\{\bdelta\}_j\cap\sigma\{\bbeta\}_k|=A_{i,j,k}$.  That is,  $\theta \fS_{\bdelta}\mapsto A$, as required.
\end{proof}
Using the fact that any conjugate of a Young subgroup is an open subgroup of $\fS_{\infty}$, it follows that any $U(\bgamma, \sigma, \bbeta)$-orbit in $X(\qrs)$ is a finitary $\wideSinfinity$-set and, hence, it makes sense to take its measure.  The following result computes this measure explicitly.

To state the result, we need to introduce the following expression.  Given $t \in R$, $q,r,s \in \MMat_{n}(t)$, and $A \in \AA_{n}(\qrs)_{t}$, define
\begin{equation}\label{E:VA-definition-with-t}
\VV_A(t) = \prod_{i,k \in [1,n]} \binom{\sum_{j \in [1,n]} A_{i,j,k}}{A_{i,1,k}, \dotsc , A_{i,n,k}} =  \prod_{i,k \in [1,n]} \frac{(\sum_{j \in [1,n]} A_{i,j,k})!}{\prod_{j \in [1,n]} (A_{i,j,k}!)}.
\end{equation}  Going forward we assume $R$ and $t$ are such that $\VV_{A}(t) \in R$ for all such arrays (e.g., $R=\k$).

\begin{lemma}\label{L:measure-computation}
Let $\mathcal{O}_{A}$ be the $U(\bgamma, \sigma, \bbeta)$-orbit in $X(\qrs)$ associated to $A \in \AA_{n}(\qrs)_{t}$ by the bijection given in the proof of the previous lemma. Then,
\[
\mu_t\left( \mathcal{O}_{A}\right) = \VV_{A}(t).
\]
\end{lemma}

\begin{proof}  Say $\theta \fS_{\bdelta} \in X(\qrs)$ is chosen so that $\theta \fS_{\bdelta}$ maps to the array $A \in \AA_{n}(\qrs)_{t}$ in the previous proof.  For short, let $U=U(\bgamma, \sigma, \bbeta)$ and let $V=V(\bgamma, \sigma, \bbeta, \theta, \bdelta)$ be the stabilizer of $\theta \fS_{\bdelta}$ for the action of $U$ on $X(\qrs)$.  Then $\mathcal{O}_{A}$ is isomorphic to $U/V$ as $\wideSinfinity$-sets so we may instead compute the measure of $U/V$ thanks to~\cite[Definition 3.1(a)]{HarmanSnowden}.

By property~\cite[Definition 3.1(e)]{HarmanSnowden}, since the canonical map $\fS_{\infty}/V \to \fS_{\infty}/U$ is a surjective map of transitive $\fS_{\infty}$-sets with fiber $U/V$, for all $t \in \k $ we have 
\[
\mu_{t}(\fS_{\infty}/V) = \mu_{t}(U/V)  \mu_{t}(\fS_{\infty}/U).
\]  On the other hand, as discussed in \cref{Ex:Measure-Computation}, there are polynomials $p_{X}(T) \in \Q [T]$ for $X \in \{\fS_{\infty}/V, \fS_{\infty}/U, U/V \}$ such that $p_{X}(t) = \mu_{t}(X)$ for all $t \in \k$.  Combining these observations shows that 
\begin{equation}\label{E:Polynomial-Relationship}
p_{\fS_{\infty}/V}(T) = p_{U/V}(T)p_{\fS_{\infty}/U}(T) 
\end{equation}in $\Q [T]$.

Note that $V= V(\bgamma, \sigma, \bbeta, \theta, \bdelta)=  \fS_{\bgamma} \cap \sigma \fS_{\bbeta}\sigma^{-1} \cap \theta\fS_{\bdelta}\theta^{-1}$ is a Young subgroup of $\fS_{\infty}$.  From the proof of \cref{L:orbit-bijection} it follows that this Young subgroup is defined by the composition $\left( A_{i,j,k} \right) \in \Lambda(n^{3})_{t}$.  Replacing $t$ with $T$ in $A_{1,1,1}$ and applying \cref{Ex:Measure-Computation} yields
\[
p_{\fS_{\infty}/V}(T) = \binom{T}{A_{1,1,1}, A_{1,1,2}, \dots , A_{n,n,n}} = \frac{T(T-1)\cdots  (T-m +1)}{\prod_{(i,j,k) \in [1,n]^{3} \backslash \{(1,1,1) \} } (A_{i,j,k}!)},
\]  where $m = \sum_{(i,j,k) \in [1,n]^{3} \backslash \{(1,1,1) \}} A_{i,j,k}$.

Likewise, note that $U= U(\bgamma, \sigma, \bbeta) = \fS_{\bgamma} \cap \sigma \fS_{\bbeta}\sigma^{-1}$ is a Young subgroup of $\fS_{\infty}$.     Moreover, if for $(i,k) \in [1,n]^{2} $ we set $\rho_{i,k} = \sum_{j \in [1,n]} A_{i,j,k}$, then $\left( \rho_{i,k}  \right) \in \Lambda(n^{2})_{t}$ is the composition which defines this Young subgroup.  Replacing $t$ with $T$ in $\rho_{1,1}$ and applying \cref{Ex:Measure-Computation} yields
\[
p_{\fS_{\infty}/U}(T) = \binom{T}{\rho_{1,1}, \rho_{1,2}, \dots , \rho_{n,n}} = \frac{T(T-1)\cdots (T-\ell +1) }{ \prod_{(i,k) \in [1,n]^{2} \backslash \{(1,1) \}} (\rho_{i,k}!)},
\] where $\ell = \sum_{(i,k) \in [1,n]^{2}\backslash \{(1,1) \}}\rho_{i,k}= \sum_{(i,j,k) \in [1,n]^{3} \backslash \{(1,1,1) \}} A_{i,j,k} - \sum_{j \in [2,n]} A_{1,j,1}$.

Substituting these into \cref{E:Polynomial-Relationship} and simplifying demonstrates that 
\begin{align}
p_{U/V}(T) &=   \frac{(T-\ell)(T-\ell-1)\dots (T - m + 1)}{\prod_{j \in [2,n]} (A_{1,j,1}!)} \prod_{(i,k) \in [1,n]^{2} \backslash \{(1,1) \}} \frac{\rho_{i,k}!}{A_{i,1,k}!\dots A_{i,n,k}!} \nonumber \\
          &=  \binom{T-\ell}{T-m,A_{1,2,1}, \dots , A_{1,n,1}}  \prod_{(i,k) \in [1,n]^{2} \backslash \{(1,1) \}} \binom{ \rho_{i,k} }{A_{i,1,k}, \dots , A_{i,n,k}}\label{E:Integer-valued-VAT}\\
          &= \VV_{A}(T). \nonumber
\end{align}  Evaluating $p_{U/V}(T)$ at $T=t$ proves the assertion.
\end{proof}

\begin{remark}\label{R:VAT-is-an-Integer-Valued-Polynomial} If $R = \Q [T]$ and $t=T$, the formulation of $\VV_{A}(T)$ given in \cref{E:Integer-valued-VAT} along with the discussion in \cref{SS:RingofIntegerValuedPolys} shows that $\VV_{A}(T)$ is an integer-valued polynomial.
\end{remark}

\begin{theorem}\label{T:product-formula}
Let $q, r, s \in \MMat_{n}(t)$. Then,
\[
C_{q,r}^{s}(t) =\sum_{A \in \AA_{n}(\qrs)_{t}} \VV_{A}(t).
\]
\end{theorem}
\begin{proof}  This follows immediately from the lemmas.  Namely,
\[
C_{q,r}^s(t) = \mu_{t}\left(X(\qrs) \right) = \sum_{A \in \AA_{n}(\qrs)_{t}} \mu_{t}(\mathcal{O}_{A}) =  \sum_{A \in \AA_{n}(\qrs)_{t}} \VV_{A}(t).
\]
\end{proof}

Recall that for $q \in \MMat_{n}(\bdelta, \bbeta)_{t}$, $q^{*} \in \MMat_{n}(\bbeta , \bdelta )_{t}$ denotes the transpose matrix.

\begin{corollary}\label{C:involutive-anti-automorphism-on-Snt}
    The map $\tau:\SS(n,t)\to \SS(n,t)$ defined by  $\xi_q\mapsto\xi_{q^*}$ is an involutive anti-automorphism of $\k$-algebras.
\end{corollary}
\begin{proof} Let $A = (A_{i,j,k}) \in \AA_{n}(\qrs)_{t}$.  There is a bijection $\AA_{n}(\qrs)_{t} \to \AA_{n}(r^{*}, q^{*}, s^{*})_{t}$ given by $A = (A_{i,j,k}) \mapsto A^{*}=(A_{k,j,i})$. The formula given in \cref{E:VA-definition-with-t} can be used to show that $\VV_{A}(t) = \VV_{A^{*}}(t)$ for all $A \in \AA_{n}(\qrs)_{t}$ and, hence, $C_{q,r}^{s}(t)=C_{r^{*}, q^{*}}^{s^{*}}(t)$ for all $\qrs \in \MMat_{n}(t)$.  The claim follows.
\end{proof}

Given a composition $\bbeta = (\beta_{1}, \dots , \beta_{n}) \in \Lambda(n)_{t}$, let $q_{\bbeta} = (q_{i,j}) \in \MMat_{n}(\bbeta, \bbeta)_{t}$ be the matrix which has $q_{i,i}= \beta_{i}$ and $q_{i,j}=0$ for $i \neq j$.  An elementary application of the formula for $C_{q,r}^{s}(t)$ establishes the following result.

\begin{corollary}\label{C:Identity-Morphisms} 
For each $\bbeta \in \Lambda(n)_{t}$ the element $1_{\bbeta} := \xi_{q_{\bbeta}} \in \Hom_{\Perm (\fS_{\infty}; \mu_{t})}(\MM^{\bbeta}, \MM^{\bbeta})$ is the identity morphism.  Thus, $\SS(n,t)$ is a locally unital $\k$-algebra with distinguished set of pairwise orthogonal idempotents $\left\{1_{\bbeta} \mid \bbeta \in \Lambda(n)_{t} \right\}$. These idempotents are fixed by the anti-involution $\tau$.
\end{corollary}

\section{Negligible Morphisms, Semisimplification, and the Classical Schur Algebra}\label{S:Semisimplification}

\subsection{Negligible Morphisms}\label{SS:negligible-morphisms}

By~\cite[Corollary 8.7]{HarmanSnowden}, $\Perm (\fS_{\infty}; \mu_{t})$ is a $\k$-linear rigid tensor category. The tensor product on objects is given by the direct product of finitary $\fS_{\infty}$-sets.  Every object of $\Perm (\fS_{\infty}; \mu_{t})$ is self-dual and $\Perm (\fS_{\infty}; \mu_{t})$ is spherical.  We refer the reader to ~\cite{EGNO} for definitions and standard results for such categories.

Such categories have a $\k$-linear \emph{categorical trace} map
\[
\tr=\tr_{X}: \End_{\Perm (\fS_{\infty}; \mu_{t})}(X) \to \k
\] for every object $X$ in $\Perm (\fS_{\infty}; \mu_{t})$.  In this setting the categorical trace has a concrete description~\cite[Proposition 8.10]{HarmanSnowden}.  Namely, given $f \in \End_{\Perm (\fS_{\infty}; \mu_{t})}(X)$ (i.e., a Schwartz function $f: X \times X \to \k$), then 
\begin{equation}\label{E:trace-definition}
\tr (f) = \int_{X} f(x,x) \;\; dx.
\end{equation}

A morphism $f : X \to Y$ is called \emph{negligible} if for all morphisms $g: Y \to X $, $\tr (f \circ g)=\tr (g \circ f) =0$.  For every pair of objects $X, Y$ in $\Perm (\fS_{\infty}; \mu_{t})$, let
\[
\NN (X,Y) = \left\{f \in \Hom_{\Perm (\fS_{\infty}; \mu_{t})}(X,Y) \mid \text{$f$ is negligible} \right\}.
\]  By \cite[Exercise 8.18.9]{EGNO} or \cite[Proposition 2.4]{EO-semisimplification}, $\NN$ is a tensor ideal of $\Perm (\fS_{\infty}; \mu_{t})$ and the quotient category $\Perm (\fS_{\infty}; \mu_{t})/\NN$ is a semisimple $\k$-linear rigid tensor category.  This is known as the \emph{non-degenerate quotient} or \emph{semisimplification} of $\Perm (\fS_{\infty}; \mu_{t})$.  

Recall that $\xi_{q_{\bbeta}}$ is the identity morphism in $\End_{\Perm (\fS_{\infty}; \mu_{t})}(\MM^{\bbeta})$.

\begin{lemma}\label{L:basic-trace-facts}  In $\Perm (\fS_{\infty}; \mu_{t})$ the following categorical trace values hold:
\begin{enumerate}
\item Given $q \in \MMat_{n}(\bbeta, \bbeta)_{t}$, 
\[
\tr (\xi_{q}) = \delta_{q, q_{\bbeta}} \binom{t}{\beta_{1}, \dotsc , \beta_{n}};
\]
\item Given $q \in \MMat_{n}(\bbeta, \bdelta)_{t}$ and $r \in \MMat_{n}(\bdelta, \bbeta)_{t}$, let $m = \sum_{(i,j) \in [1,n]^{2}\backslash \{(1,1) \}} q_{i,j}$. Then there exists $z \in \Q_{>0 }$ such that
\begin{equation}\label{E:trace-calculation}
\tr (\xi_{q} \circ \xi_{r}) = \delta_{r,q^{*}} \prod_{i \in [1,n]} \frac{\left(\sum_{j \in [1,n]} q_{i,j} \right)!}{q_{i,1}!q_{i,2}!\dotsb q_{i,n}!}
      = \delta_{r,q^{*}}t(t-1)(t-2)\dotsb (t-m+1)z,
\end{equation}
\end{enumerate}
\end{lemma}
\begin{proof} To prove (1), choose $\sigma \in \fS_{\infty}$ so that the $\fS_{\infty}$-orbit in $\fS_{\infty}/\fS_{\bbeta} \times \fS_{\infty}/\fS_{\bbeta}$ which contains $(\fS_{\bbeta}, \sigma \fS_{\bbeta})$ corresponds to $q$ as described in \cref{SS:Infinite-Symmetric-Group}.  For short, write $\mathcal{O}_{\sigma}$ for that orbit. We compute the trace using \cref{E:trace-definition}, 
\[
\tr (\xi_{q}) = \int_{\fS_{\infty}/\fS_{\bbeta}} \xi_{q}(\theta \fS_{\bbeta},\theta \fS_{\bbeta}) \;\; d\left( \theta \fS_{\bbeta}\right).
\]  Since $\xi_{q}$ is the indicator function for $\mathcal{O}_{\sigma}$ and since $(\theta \fS_{\bbeta}, \theta \fS_{\bbeta})$ is in $\mathcal{O}_{\sigma}$ if and only if $\sigma \in \fS_{\bbeta}$, this integral is identically zero unless $q = q_{\bbeta}$.  In that case, 
\begin{align*}
\tr (\xi_{q_{\bbeta}}) &= \int_{\fS_{\infty}/\fS_{\bbeta}} \xi_{q_{\bbeta}}(\theta \fS_{\bbeta},\theta \fS_{\bbeta}) \;\; d\left( \theta \fS_{\bbeta} \right)\\
    &= \mu_{t}\left(\fS_{\infty}/\fS_{\bbeta} \right) \\
    &= \binom{t}{\beta_{1}, \dotsc , \beta_{n}},
\end{align*} where the last equality was given in \cref{Ex:Measure-Computation}.

To prove (2), note that the $\k$-linearity of the trace function along with (1) immediately implies that 
\[
\tr (\xi_{q} \circ \xi_{r}) = C_{q,r}^{q_{\bbeta}}(t)\tr (\xi_{q_{\bbeta}}),
\] where $C_{q,r}^{q_{\bbeta}}(t)$ is the structure constant described by \cref{T:product-formula}.     We next claim that $\AA_{n}(q,r, q_{\bbeta})$ contains exactly one element if $r = q^{*}$ and is otherwise empty.  Let $A = (A_{i,j,k})$ be an array in $\AA_{n}(q,r,q_{\bbeta})$.  From the fact that the entries of $A$ are nonnegative integers (excepting $A_{1,1,1}$, of course), the condition that $\sum_{j \in [1,n]} A_{i,j,k} = (q_{\bbeta})_{i,k} = 0$ for $i \neq k$ implies that if $i \neq k$, then $A_{i,j,k}=0$ for all $j \in [1,n]$.  But then, since $\sum_{k \in [1,n]} A_{i,j,k} = q_{i,j}$, it follows that for all $i,j \in [1,n]$ we have $A_{i,j,k}= q_{i,j}$ if $k = i$, and $A_{i,j,k}=0$ if $k \neq i$.  That is, $A$ is completely determined by $q$ and $q_{\bbeta}$, if it exists.  Hence, there is at most one such array.  And if the array exists, then $r_{j,k} = \sum_{i \in [1,n]} A_{i,j,k} = A_{k,j,k}= q_{k,j}$ and so $r = q^{*}$.  In particular, $\tr (\xi_{q} \circ \xi_{r}) =0$ if $r \neq q^{*}$.

Now assume $r=q^{*}$ and set $A \in \AA_{n}(q,r, q_{\bbeta})$ to be the unique array determined in the previous paragraph.  Applying \cref{T:product-formula}, \cref{E:VA-definition-with-t}, and claim (1)  yields
\begin{align*}
\tr (\xi_{q} \circ \xi_{r}) &= \VV_{A}(t)\tr (\xi_{q_{\bbeta}})\\
 &=\prod_{i,k \in [1,n]} \frac{\left(\sum_{j \in [1,n]} A_{i,j,k} \right)!}{A_{i,1,k}!A_{i,2,k}!\dotsb A_{i,n,k}!}\tr (\xi_{q_{\bbeta}})\\
   &=\prod_{i \in [1,n]} \frac{\left(\sum_{j \in [1,n]} q_{i,j} \right)!}{q_{i,1}!q_{i,2}!\dotsb q_{i,n}!}\binom{t}{\beta_{1}, \dotsc , \beta_{n}}.
\end{align*}
Isolating the terms which contain $t$ yields
\begin{equation}\label{E:product-equation2}
 (t-\ell)(t-\ell-1)(t-\ell-2)\dotsb (t-m+1)t(t-1)(t-2)\dotsb (t-u+1)z,
\end{equation}
where $z \in \Q_{>0 }$, $\ell =\sum_{(i,j) \in [1,n]^{2}\backslash \{(1,1) \}} q_{i,j}- \sum_{j \in [2,n]} q_{1,j}$, $m = \sum_{(i,j) \in [1,n]^{2}\backslash \{(1,1) \}} q_{i,j}$, and $u = \sum_{i \in [2,n]} \beta_{i}$. Using $\sum_{i,j \in [1,n]} q_{i,j} = t$ and the row-sum $\sum_{j \in [1,n]} q_{1,j} = \beta_{1}$, both $\ell$ and $u$ equal $t - \beta_{1}$. Therefore, after replacing $\ell$ with $u$ and rearranging terms, \cref{E:product-equation2} becomes 
\[
t(t-1)(t-2)\dotsb (t-u+1)(t-u)\dotsb (t-m+1) z.
\] The claim follows.
\end{proof}

\begin{corollary}\label{C:Semisimplification-and-finite-Snd}  Let $q = (q_{i,j}) \in \MMat_{n}(\bdelta, \bbeta)_{t}$.  Then, $\xi_{q}$ is negligible if and only if $t \in \Z_{\geq 0}$ and $q_{1,1}<0$.
\end{corollary}

\begin{proof} Let $\ell =\sum_{(i,j) \in [1,n]^{2}\backslash \{(1,1) \}} q_{i,j}- \sum_{j \in [2,n]} q_{1,j}$ and $m = \sum_{(i,j) \in [1,n]^{2}\backslash \{(1,1) \}} q_{i,j}$.  It follows from \cref{L:basic-trace-facts} that the morphism $\xi_{q}$ is negligible if and only if $t$ is a root of \cref{E:trace-calculation}.  That is,
\[
t=0,1,\dotsc , m-1.
\] Or, equivalently, if and only if $t$ is a nonnegative integer and $t-m = q_{1,1}$ is negative. 
\end{proof}

\subsection{ \texorpdfstring{$\SS(n,d)$}{SS(n,d)} and  \texorpdfstring{$S(n,d)$}{S(n,d)}}\label{SS:SSnd-and-Snd}

Throughout this section fix $n \in \N  \cup \{\infty \}$, $d\in \Z_{\geq 0}$, and assume that $t=d$.   Recall from \cref{R:finite-inside-of-t-shifted} that $\Lambda(n,d)$ is a subset of $\Lambda(n)_{d}$,  $\Mat(n,d)$ is a subset of $\MMat_{n}(d)$, and $\A(n, \qrs)$ is a subset of $\AA_{n}(\qrs)_{d}$.

Let $\Idealnd$ be the two-sided ideal in $\SS(n,d)$ of negligible morphisms.  It follows from \cref{L:basic-trace-facts} that $\Idealnd$ is the $\k$-span of the set
\begin{equation*}
K(n,d):=\left\{\xi_{q} \mid q = (q_{i,j}) \in \MMat_{n}(d), q_{1,1} < 0 \right\}.
\end{equation*}

\begin{theorem}\label{T:SSmodIdeal-is-isomorphic-to-Snd}  Let $n \in \N \cup \{\infty \}$ and $d \in \Z_{\geq 0}$.  Then 
\[
\SS(n,d)/\Idealnd \cong S(n,d)
\] as $\k$-algebras.
\end{theorem}

\begin{proof} By construction, $\SS(n,d)/\Idealnd$ has a basis indexed by the set 
\[
M:=\left\{q = (q_{i,j}) \in \MMat_{n}(d) \mid  q_{1,1} \geq 0 \right\}.
\]
  However, this set coincides with $\Mat(n,d)$ and so also indexes the basis for $S(n,d)$.  Define a vector space isomorphism
\[
\SS(n,d)/\Idealnd \to S(n,d) 
\] by $\xi_{q} + \Idealnd \mapsto \xi_{q}$ for all $q \in M$.

Let $s \in \Mat(n,d)$, so $s_{1,1} \geq 0$.  The set $\AA_{n}(q, r, s)_{d}$ splits into arrays $A = (A_{i,j,k})$ of two types: those with $A_{1,1,1} < 0$ and those with $A_{1,1,1} \geq 0$.  For arrays with $A_{1,1,1} < 0$, the leading factor of the formula given in \cref{E:Integer-valued-VAT} is the falling factorial $(d-\ell)(d-\ell-1)\dotsb(d-m+1)$ with $d-\ell = s_{1,1} \geq 0 > A_{1,1,1} = d-m$; the product passes through zero and so $\VV_{A}(d) = 0$.  On the other hand, the subset of arrays with $A_{1,1,1} \geq 0$  identifies with $\A(n, q, r, s)$ and, moreover, $\VV_{A}(d)=v_{A}$.  These observations imply that 
\[
C_{q,r}^{s}(d) = \sum_{\substack{A \in \AA_{n}(q, r, s)_{d} \\ A_{1,1,1} \geq 0}} \VV_{A}(d) = \sum_{A \in \A(n, \qrs)} v_{A} = c_{q,r}^{s}.
\]  Therefore, this map is also a $\k$-algebra isomorphism.
\end{proof}
To establish notation used in the sequel, for a characteristic zero field $\k$ and $n,d \in \N$, let 
\[
f_{n,d}=f_{n,d}^{\k}: \SS_{\k }(n,d) \to S_{\k }(n,d)
\] be the surjective $\k$-algebra homomorphism given by $\xi_{q} \mapsto \xi_{q}$ if $q_{1,1} \geq 0$ and $\xi_{q} \mapsto 0$ if $q_{1,1} < 0$.

\begin{remark}\label{R:Integral-quotient}  Note that \cref{T:SSmodIdeal-is-isomorphic-to-Snd}  holds integrally.  Namely, since $C_{q,r}^{s}(T)$ is an integer-valued polynomial, $C_{q,r}^{s}(d)$ is an integer when $d \in \Z$ and the $\Z$-span of the basis $\left\{\xi_{q} \mid q \in \MMat_{n}(d) \right\}$ defines an integral form $\SS_{\Z}(n,d)$.  Likewise, the structure constants given in \cref{T:classical-product} are integral and the $\Z$-span of $\left\{\xi_{q} \mid q \in \Mat(n,d) \right\}$ defines an integral form  $S_{\Z}(n,d)$.  The restriction of $f_{n,d}^{\Q}$ defines a surjective map from $\SS_{\Z}(n,d)$ to $S_{\Z}(n,d)$ which has kernel given by the $\Z$-span of $K(n,d)$.
\end{remark}

\subsection{The algebra  \texorpdfstring{$\SS_{\IntegerValuedRing}(n,T)$}{SS(n,T)}}\label{SS:naive-Z<T>-form}

Consider $\k = \Q (T)$ and set $t=T$.  Let 
\[
\SS_{\IntegerValuedRing}(n,T) \subseteq \SS_{\Q(T)}(n,T)
\]
denote the free $\IntegerValuedRing$-submodule of $\SS_{\Q(T)}(n,T)$ with basis $\left\{\xi_{q} \mid q \in \MMat_{n}(T) \right\}$.  Since the structure constants $C_{q,r}^{s}(T)$ lie in $\IntegerValuedRing$ for any $\qrs \in \MMat_{n}(T)$, $\SS_{\IntegerValuedRing}(n,T)$ is a $\IntegerValuedRing$-form in $\SS_{\Q  (T)}(n,T)$.

Given a characteristic zero field $\k$ and $t \in \k$, there is a ring homomorphism given by evaluation:
\begin{align*}
\ev_t:\SS_{\Q[T] }(n,T)&\to \SS_{\k}(n,t),\\
\sum_{q\in \MMat_{n}(T)} p_{q}(T)\xi_{q}&\mapsto \sum_{q(t)\in \MMat_{n}(t)} p_q(t)\xi_{q(t)}.
\end{align*}

\begin{remark}\label{R:integral-evaluation-plus-quotient}  When $t=d$ is a nonnegative integer and $\k$ is a characteristic zero field, write $F_{n,d}=F_{n,d}^{\k }$ for the composite

\[
\SS_{\Q [T] }(n,T)\xrightarrow{\ev_{d}} \SS_{\k}(n,d) \xrightarrow{f^{\k}_{n,d}} S_{\k}(n,d).
\]

It will be useful for the sequel to record the fact that the family of maps $\left\{F_{n,d} \right\}_{d \in \N}$ is asymptotically injective. That is, for any fixed nonzero $x \in \SS_{\Q[T]}(n,T)$, $F_{n,d}(x)$ will be nonzero for all sufficiently large $d$.
\end{remark}

\section{Codeterminant Basis for  \texorpdfstring{$\SS_{\IntegerValuedRing}(n,T)$}{SS(n,T)}}\label{S:Codeterminant-Basis-for-SnT}

In this section we establish bases for $\SS_{\IntegerValuedRing}(n,T)$ and $\SS_{\k}(n,t)$ that parallel Green's codeterminant bases for $S_{\Z}(n,d)$ and $S_{\k}(n,d)$ discussed in \cref{SS:Finite-Codeterminant-Basis}.

\subsection{Codeterminant Index Sets with a  \texorpdfstring{$t$}{t}}\label{SS:Codeterminant-Index-Sets}

For $\blambda \in \Lambda^{+}(n)_{t}$ and $\bbeta, \bdelta \in \Lambda(n)_{t}$, let
\[
X(\blambda, \bbeta)_{t} = \left\{\xi_{r} \mid r  \in \MMat_{n}(\blambda , \bbeta )_{t} \text{ and $r$ is semistandard} \right\}
\] and
\[
X(\blambda)_{t} = \bigcup_{\bbeta \in \Lambda (n)_{t}} X(\blambda, \bbeta)_{t}. 
\] Likewise, let
\[
Y(\bdelta, \blambda)_{t} = \left\{\xi_{q} \mid q  \in \MMat_{n}(\bdelta , \blambda )_{t} \text{ and $q^{*}$ is semistandard} \right\}
\]
and 
\[
Y(\blambda)_{t} = \bigcup_{\bdelta \in \Lambda (n)_{t}} Y(\bdelta, \blambda)_{t}. 
\]

\subsection{Codeterminant Basis for  \texorpdfstring{$\SS_{\IntegerValuedRing}(n,T)$}{SS(n,T)}}\label{SS:codeterminant-basis-for-SnT}

\begin{theorem}\label{T:codeterminant-basis-Snt} Let $\k = \Q (T)$ and fix $t=T$.   Let $n \in \N \cup \{\infty \}$.  For each $\bbeta, \bdelta \in \Lambda(n)_{T}$, the set
\[
\CoDetBasis(\bdelta, \bbeta)_{T}=\left\{ \xi_{q}\xi_{r} \left|   (\xi_{q},\xi_{r})\in \bigcup_{\blambda\in\Lambda^{+}(n)_{T}}Y(\bdelta,\blambda)_{T}\times X(\blambda,\bbeta)_{T}\right. \right\}
\]
is a $\Q (T)$-basis for $1_{\bdelta}\SS_{\Q (T)}(n,T)1_{\bbeta}$. 
\end{theorem}
\begin{proof}  Recall from \cref{R:integral-evaluation-plus-quotient} that there is a ring homomorphism: 
\[
F_{n,d}^{\Q}: \SS_{\Q[T] }(n,T) \to S_{\Q}(n,d).
\]  Recalling that $F_{n,d}^{\Q}(\xi_{q})$ is zero if $q(d)_{1,1} < 0$ and is otherwise the basis element of $S_{\Q }(n,d)$ indexed by $q(d) \in \Mat(n,d)$, we abuse notation for the duration of the proof by writing $\xi_{q(d)} \in S_{\Q}(n,d)$ for $F_{n,d}^{\Q}(\xi_{q})$. 

For a fixed $\bbeta, \bdelta \in \Lambda(n)_{T}$, the sets  
\[
\bigcup_{\blambda\in\Lambda^{+}(n)_{T}}Y(\bdelta,\blambda)_{T}\times X(\blambda,\bbeta)_{T}
\]
and 
\[
\MMat_{n}(\bdelta, \bbeta)_{T}
\]
are finite.  Therefore, as discussed in \cref{R:Evaluation-at-large-d-is-a-bijection}, for all sufficiently large $d$ it is true that:
\begin{enumerate} 
\item the map
\[
\bigcup_{\blambda\in\Lambda^{+}(n)_{T}}Y(\bdelta,\blambda)_{T}\times X(\blambda,\bbeta)_{T} \to \bigcup_{\blambda\in\Lambda^{+}(n,d)}Y(\bdelta(d),\blambda(d))\times X(\blambda(d),\bbeta(d)) 
\] given by $(\xi_{q},\xi_{r}) \mapsto (\xi_{q(d)}, \xi_{r(d)})$ is a bijection;
\item the map
\[
\MMat_{n}(\bdelta, \bbeta)_{T} \to \Mat(n, d, \bdelta(d), \bbeta(d))
\] given by $q \mapsto q(d)$ is a bijection;
\item the filling of the Young diagram $Y^{\blambda(d)}$ by $r(d)$ and by $q^{*}(d)$ are both semistandard for all $(\xi_{q}, \xi_{r})$ in $\bigcup_{\blambda\in\Lambda^{+}(n)_{T}}Y(\bdelta,\blambda)_{T}\times X(\blambda,\bbeta)_{T}$.
\end{enumerate}
Under those conditions, 
\begin{align*}
\left|   \bigcup_{\blambda\in\Lambda^{+}(n)_{T}}Y(\bdelta,\blambda)_{T}\times X(\blambda,\bbeta)_{T} \right| & = \left| \bigcup_{\blambda\in\Lambda^{+}(n,d)}Y(\bdelta(d),\blambda(d))\times X(\blambda(d),\bbeta(d)) \right|, \\
\left| \MMat_{n}(\bdelta, \bbeta)_{T} \right| & =\left|  \Mat(n, d,  \bdelta(d), \bbeta(d)) \right|,
\end{align*}
 and the set
\[
\CB(n, d, \bdelta(d), \bbeta(d)) = \left\{ \xi_{q(d)}\xi_{r(d)} \left|   (\xi_{q(d)},\xi_{r(d)})\in \bigcup_{\blambda(d)\in\Lambda^{+}(n,d)}Y(\bdelta(d),\blambda(d))\times X(\blambda(d),\bbeta(d))\right. \right\}
\] is Green's codeterminant basis for $\Hom_{\Q \fS_{d}}(M^{\bbeta(d)}, M^{\bdelta(d)}) = 1_{\bdelta (d)}S_{\Q}(n,d)1_{\bbeta (d)}$.  

To show linear independence, let $\xi_{q_{1}}\xi_{r_{1}},\ldots,\xi_{q_{p}}\xi_{r_{p}}$ be a finite collection of distinct elements from $\CoDetBasis(\bdelta, \bbeta)_{T}$ and say 
\begin{equation}\label{E:A-linear-dependency-of-codeterminants}
\sum_{i=1}^{p} r_i(T)\xi_{q_{i}}\xi_{r_{i}}=0,
\end{equation}
where $r_i(T)=a_i(T)/b_i(T)$ for some $a_{i}(T), b_{i}(T) \in \Q [T]$ with $b_{i}(T) \neq 0$. By clearing denominators we may assume without loss of generality that $r_{i}(T) \in \Q [T]$ for all $i \in [1,p]$ and, hence, \cref{E:A-linear-dependency-of-codeterminants} lies in $\SS_{\Q [T]}(n,T)$. Applying $F_{n,d}^{\Q}$ to \cref{E:A-linear-dependency-of-codeterminants} yields 
\[
\sum_{i=1}^{p} r_i(d)\xi_{q_{i}(d)}\xi_{r_{i}(d)}=0,
\] in $S_{\Q}(n,d)$ for all $d \gg 0$.  However, $\xi_{q_{i}(d)}\xi_{r_{i}(d)}$ are distinct elements of Green's codeterminant basis. Therefore, $r_{i}(d)=0$ for all $i \in [1,p]$ and all $d \gg 0$. Thus, $r_i(T)=0$ for all $i \in [1,p]$, as required.

For spanning, since $1_{\bdelta}\SS_{\k}(n,T)1_{\bbeta}$ is finite-dimensional, it suffices to observe that for $d \gg 0$ we have:
\begin{align*}
             \dim_{\k} \left(1_{\bdelta}\SS_{\k }(n,T)1_{\bbeta} \right)         & = |\MMat_{n}(\bdelta, \bbeta)_{T}| \\
                      & = |\Mat(n, d, \bdelta(d), \bbeta(d))| \\
		      &= \dim_{\Q}\left( 1_{\bdelta (d)}S_{\Q}(n,d)1_{\bbeta (d)}\right)\\
                      & = |\CB (n, d, \bdelta(d), \bbeta(d))| \\
                      & = | \CoDetBasis (\bdelta, \bbeta)_{T}|.
\end{align*} 
\end{proof} 
\begin{remark}
While Green stated his results for the case when $n$ is finite, his and our arguments go through unchanged if $n=\infty$. 
\end{remark}

\begin{corollary}\label{C:Codeterminant-basis-for-SnT}  The set 
\begin{equation*}
\CoDetBasis(n)_{T} = \bigcup_{\blambda \in \Lambda^{+}(n)_{T}}\left\{\xi_{q}\xi_{r} \left| (\xi_{q}, \xi_{r}) \in Y(\blambda)_{T} \times X(\blambda)_{T}  \right. \right\}
\end{equation*}
 is a basis for $\SS_{\Q (T)}(n,T)$.
\end{corollary}

The next result shows that, perhaps surprisingly, the change of basis matrix from $\CoDetBasis (n)_{T}$ to $\StdBasis(n)_{T}$ is defined over the integers.

\begin{proposition}\label{P:codeterminants-are-integral-combos-of-indicator-functions}  Let $\k= \Q (T)$.  Let $n \in \N  \cup \{\infty \}$. Say $\xi_{q}\xi_{r} \in \CoDetBasis (\bdelta, \bbeta)_{T}$.  If one writes
\[
\xi_{q}\xi_{r} = \sum_{s \in \MMat_{n}(\bdelta, \bbeta)_{T}} C_{q,r}^{s}(T) \xi_{s}
\] in $\SS_{\k}(n,T)$, then $C_{q,r}^{s}(T) \in \Z$ for all $s \in \MMat_{n}(\bdelta, \bbeta)_{T}$.
\end{proposition}
\begin{proof}
Since $\xi_{q} \in Y(\bdelta, \blambda)_{T}$, the discussion after \cref{E:right-codeterminant-condition} shows that $q$ is lower triangular. Let $s \in \MMat_{n}(\bdelta, \bbeta)_{T}$ and let $A$ be an array in $\AA_{n}(\qrs)_{T}$.  Since $\sum_{k \in [1,n]} A_{i,j,k} = q_{i,j}=0$ for $i<j$, this implies $A_{i,j,k}=0$ for all $k \in [1,n]$ as long as $i<j$.  In particular, $A_{1,j,1}=0$ for $j \in [2,n]$.  But then $\sum_{j \in [2,n]} A_{1,j,1} = 0$ and $\ell = m$, so the leading binomial $\binom{T-\ell}{T-m, 0, \dots, 0}$ in the formula for $\VV_{A}(T)$ given in \cref{E:Integer-valued-VAT} equals one.  Hence,  $\VV_{A}(T)$ is an integer.  This fact along with \cref{T:product-formula} proves the claim.
\end{proof}

 Let
\begin{equation}\label{E:TnT-definition}
\TT_{\IntegerValuedRing}(n,T) \subseteq \SS_{\Q (T)}(n,T)
\end{equation} be the $\IntegerValuedRing$-span of $\CoDetBasis(n)_{T}$.   The next result shows this is  a $\IntegerValuedRing$-form for $\SS_{\Q (T)}(n,T)$.

\begin{proposition}\label{P:TnT-is-an-integral-form} For all $n \in \N \cup \{\infty \}$, $\TT_{\IntegerValuedRing}(n,T)$ is a $\IntegerValuedRing$-algebra.
\end{proposition}
\begin{proof} All that needs to be verified is that $\TT_{\IntegerValuedRing}(n,T)$ is closed under multiplication.  Let  $\xi_{q_{1}}\xi_{r_{1}} \in \CoDetBasis (\bgamma,\bdelta)_{T}$ and  $\xi_{q_{2}}\xi_{r_{2}} \in \CoDetBasis (\bdelta,\bbeta)_{T}$.  Since $\TT_{\IntegerValuedRing}(n,T)$ is contained in $\SS_{\IntegerValuedRing}(n,T)$, a $\IntegerValuedRing$-algebra, we have
\begin{equation}\label{E:yet-another-formula}
(\xi_{q_{1}}\xi_{r_{1}})(\xi_{q_{2}}\xi_{r_{2}}) = \sum_{s \in \MMat_{n}(\bgamma, \bbeta)_{T}} a_{s}(T)\xi_{s},
\end{equation}
for $a_{s}(T) \in \IntegerValuedRing$.   \cref{P:codeterminants-are-integral-combos-of-indicator-functions} implies that each $\xi_{s}$ can be written as a $\Q$-linear combination of elements from $\CoDetBasis (\bgamma, \bbeta)_{T}$.  Substituting these combinations into \cref{E:yet-another-formula} yields
\[
(\xi_{q_{1}}\xi_{r_{1}})(\xi_{q_{2}}\xi_{r_{2}}) = \sum_{\xi_{t}\xi_{u} \in \CoDetBasis (\bgamma, \bbeta)_{T}} b_{t,u}(T)\xi_{t}\xi_{u},
\] for $b_{t,u}(T) \in \Q [T]$.  Taking $d \gg 0$ and applying $F_{n,d}^{\Q}$ yields 
\[
(\xi_{q_{1}(d)}\xi_{r_{1}(d)})(\xi_{q_{2}(d)}\xi_{r_{2}(d)}) = \sum_{\xi_{t}\xi_{u} \in \CoDetBasis (\bgamma, \bbeta)_{T}} b_{t,u}(d) \xi_{t(d)}\xi_{u(d)}
\] in $S_{\Q}(n,d)$.

Green's codeterminant basis has integral structure constants \cite{Green-codets,Woodcock}.  That is, $b_{t,u}(d)$ is an integer for all $d \gg 0$ and $b_{t,u}(T) \in \IntegerValuedRing$ for all $\xi_{t}\xi_{u} \in \CoDetBasis (\bgamma, \bbeta)_{T}$, as desired.
\end{proof}

Let $\k$ be a characteristic zero field and fix $t \in \k$.  Then $\k$ is a $\IntegerValuedRing$-module via the evaluation map $\ev_{t}:\IntegerValuedRing \to\k$ given by $\ev_{t}(p(T))=p(t)$.  There are canonical $\k$-algebra isomorphisms 
\[
\SS_{\IntegerValuedRing}(n,T) \otimes_{\IntegerValuedRing} \k \cong \SS_{\k}(n,t)\cong \TT_{\IntegerValuedRing}(n,T) \otimes_{\IntegerValuedRing} \k .
\]  In particular, via these isomorphisms it follows that the set 
\begin{equation}\label{E:Codet-Basis-for-Snt}
\CoDetBasis (n)_{t}=\left\{ \xi_{q}\xi_{r} \left| (\xi_{q}, \xi_{r}) \in \cup_{\blambda \in \Lambda^{+}(n)_{t}} Y(\blambda)_{t} \times X(\blambda)_{t} \right. \right\}
\end{equation}
defines a $\k$-basis for $\SS_{\k}(n,t)$.

\section{Representation Theory of \texorpdfstring{$\SS (n,t)$}{SS(n,t)}}\label{S:Snt-is-quasi-hereditary}

In this section we establish that $\SS (n,t)$ is a symmetrically based quasi-hereditary algebra and, hence, its representations define a highest weight category.

\subsection{ \texorpdfstring{$\SS (n,t)$}{SS(n,t)}-Modules}\label{SS:Snt-rep-theory-notations}  

Let $t \in \k$ and $n \in \N  \cup \{\infty \}$. Recall that $\SS(n, t)$ is a locally unital algebra with distinguished idempotents $\{1_{\bbeta} \mid \bbeta \in \Lambda(n)_{t} \}$. By \cref{P:Snt-Basis-Theorem}, it is \emph{locally finite-dimensional}: $1_{\bdelta}\SS(n,t)1_{\bbeta}$ is a finite-dimensional $\k$-vector space for all $\bbeta, \bdelta \in \Lambda(n)_{t}$.

   By definition, an $\SS(n,t)$-module is a module, $M$, for the associative $\k$-algebra $\SS(n,t)$ which also satisfies $M = \bigoplus_{\bbeta \in \Lambda(n)_{t}} 1_{\bbeta}M$, where the direct sum is a direct sum of $\k$-vector spaces.  We call elements of $M_{\bbeta}:=1_{\bbeta}M$ \emph{weight vectors} of \emph{weight} $\bbeta$ and call $M_{\bbeta}$ a \emph{weight space}.  We call $M$ locally finite-dimensional if every weight space of $M$ is a finite-dimensional $\k$-vector space.  Let $\Sntmod$ denote the category of all $\SS(n,t)$-modules, let $\Sntmodlfd$ denote the full subcategory of locally finite-dimensional $\SS(n,t)$-modules, and let $\Sntmodfg$ denote the full subcategory of finitely generated $\SS(n,t)$-modules.  Note that finitely generated modules for a locally finite-dimensional algebra are always locally finite-dimensional.

There is a contravariant equivalence between right and left locally finite-dimensional $\SS(n,t)$-modules given by the graded dual:  $M \mapsto M^{\circledast} = \bigoplus_{\bbeta \in \Lambda(n)_{t}} M_{\bbeta}^{*}$, where $M_{\bbeta}^{*} = \Hom_{\k}(M_{\bbeta}, \k)$.   There is a contravariant autoequivalence of $\Sntmodlfd$ given by $M \mapsto M^{\vee}$, where $M^{\vee} = M^{\circledast}$ as a $\k$-vector space and the action of $\SS(n,t)$ given by $(a.f)(m) = f(\tau(a)m)$ for all $f \in M^{\vee}$, $a \in \SS(n,t)$, and $m \in M$, where $\tau$ is the involutive anti-automorphism defined in \cref{C:involutive-anti-automorphism-on-Snt}.

\subsection{Highest Weight Combinatorics}\label{SS:Snt-quasi-hereditary-combinatorics}

Recall the dominance partial order on $\Lambda(n)_{t}$ given by \cref{E:dominance-order-definition}.  Call a partially ordered set $(X, \leq )$ \emph{upper finite} if $[x,\infty):=\{y \in X \mid x \leq  y \}$ is a finite set for all $x \in X$.  The following lemma records some of the basic properties of $(\Lambda(n)_{t}, \preceq)$.

\begin{lemma}\label{L:upper-finite-axioms}  The following statements are true:
\begin{enumerate}
\item The partially ordered set $(\Lambda(n)_{t},\preceq)$ is upper finite.
\item For $\blambda \in\Lambda^{+}(n)_{t}$ and $\bbeta \in \Lambda(n)_{t}$, the sets $X(\blambda,\bbeta)_{t}$ and $Y(\bbeta,\blambda)_{t}$ are empty unless $\bbeta\preceq\blambda$.
\item For $\blambda\in\Lambda^{+}(n)_{t}$, $X(\blambda,\blambda)_{t}=Y(\blambda,\blambda)_{t}=\{\xi_{q_{\blambda}}=1_{\blambda}\}$.
\end{enumerate}
\end{lemma}

\begin{proof}
To prove (1), let $\bmu = (\mu_{1}, \dotsc , \mu_{n}) \in \Lambda(n)_{t}$ and consider $\blambda = (\lambda_{1}, \dotsc , \lambda_{n})\in[\bmu,\infty)$. Then 
\[
\left(t- \sum_{i \in [2,n]} \lambda_{i} \right) - \left(t- \sum_{i \in [2,n]} \mu_{i} \right) = \lambda_{1} - \mu_{1} \geq  0.
\]
Consequently,
\[
\sum_{i \in [2,n]}\mu_{i} \geq    \sum_{i \in [2,n]}\lambda_{i}.
\]
However, $\lambda_{2}, \dotsc , \lambda_{n}$ are nonnegative integers and only finitely many $\blambda\in\Lambda(n)_{t}$  satisfy this condition for a fixed $\bmu$.

To prove (2), consider $q \in X(\blambda, \bbeta)_{t}$ and $d \gg 0$.  Since $q(d)$ is upper triangular, for each $p \in [1,n]$ it is true that
\[
\sum_{j=1}^{p} \beta(d)_{j} = \sum_{i=1}^{n} \sum_{j=1}^{p} q(d)_{i,j}=\sum_{i=1}^{p} \sum_{j=1}^{p} q(d)_{i,j} \leq \sum_{i=1}^{p}\sum_{j=1}^{n} q(d)_{i,j} = \sum_{i=1}^{p} \lambda(d)_{i}.
\]  Therefore $\bbeta \preceq \blambda$, as claimed.  A similar argument applies to $Y(\bbeta, \blambda)_{t}$.

To prove (3), observe that there is one and only one filling of $Y^{\blambda(d)}$ by a $q \in \MMat_{n}(\blambda, \blambda)_{t}$ which is semistandard.  Namely, since $\sum_{i \in [1,n]} q(d)_{i,1} = \blambda(d)_{1}$, a filling of $Y^{\blambda(d)}$ by $q(d)$ will have $\blambda(d)_{1}$ ones. For the filling to be semistandard, they must exactly be used to fill the $\blambda(d)_{1}$ boxes in the first row of $Y^{\blambda(d)}$.  Similarly, $\sum_{i \in [1,n]} q(d)_{i,2} = \blambda(d)_{2}$ and a filling of $Y^{\blambda(d)}$ by $q(d)$ will have $\blambda(d)_{2}$ twos. For the filling to be semistandard, they must exactly be used to fill the $\blambda(d)_{2}$ boxes in the second row of $Y^{\blambda(d)}$.  And so on.  Thus $q = q_{\blambda}$, as claimed.  Likewise, for $q^{*}$.
\end{proof}

\subsection{ \texorpdfstring{$\SS (n,t)$}{SS(n,t)}  is an upper finite symmetrically based quasi-hereditary algebra}\label{SS:Snt-is-based-quasi-hereditary}

\begin{theorem}\label{T:Snt-is-based-qh}  Let $\k$ be a characteristic zero field, fix $t \in \k$, and fix $n \in \N \cup \{\infty \}$.  Consider $\SS(n,t)$ with the anti-involution $\tau$ defined in \cref{C:involutive-anti-automorphism-on-Snt}.  Set $I= \Lambda(n)_{t}$ and $\Lambda = \Lambda^{+}(n)_{t}$, both equipped with the dominance order $\preceq$.  For $\bmu  \in \Lambda(n)_{t}$ and $\blambda \in \Lambda^{+}(n)_{t}$, let $X(\blambda, \bmu )_{t}$, $X(\blambda)_{t}$, $Y(\bmu , \blambda)_{t}$, and  $Y(\blambda)_{t}$  be as in \cref{SS:Codeterminant-Index-Sets}.

This data makes $\SS(n,t)$ an upper finite symmetrically based quasi-hereditary algebra in the sense of~\cite[Definition 5.1]{BrundanStroppel}.
\end{theorem}
\begin{proof} The reader can easily confirm that the six defining properties of an upper finite, based quasi-hereditary algebra given in \cite[Definition 5.1]{BrundanStroppel} have already been established.  Namely, $\SS (n,t) = \bigoplus_{\bdelta, \bbeta \in I} 1_{\bdelta} \SS (n,t) 1_{\bbeta}$ is a locally unital $\k$-algebra which satisfies:
\begin{enumerate}
\item there is a subset of special idempotents $\left\{1_{\blambda} \mid \blambda \in \Lambda \right\}$ for a fixed subset $\Lambda \subseteq I$;
\item  $\Lambda$ is upper finite with respect to $\preceq$ (see \cref{L:upper-finite-axioms});
\item there exist subsets $X(\blambda, \bbeta)_{t} \subseteq 1_{\blambda}\SS (n,t) 1_{\bbeta}$ and $Y(\bdelta , \blambda)_{t} \subseteq 1_{\bdelta} \SS (n,t) 1_{\blambda}$ for all $\blambda \in \Lambda$ and all $\bdelta \in I$;
\item  the products $yx$ over all $ (y,x) \in \bigcup_{\blambda \in \Lambda}  Y(\blambda )_{t} \times X(\blambda )_{t}$ form a basis for $\SS (n,t)$ (see \cref{E:Codet-Basis-for-Snt});
\item for $\blambda, \bmu \in \Lambda$, the sets $X(\blambda, \bmu )_{t}$ and $Y(\bmu , \blambda)_{t}$ are empty unless $\bmu  \preceq \blambda$ (see \cref{L:upper-finite-axioms});
\item  $X(\blambda, \blambda )_{t}= Y(\blambda , \blambda)_{t} = \{ 1_{\blambda}\}$ for each $\blambda \in \Lambda$ (see \cref{L:upper-finite-axioms}).
\end{enumerate}

All that remains is to verify that $\SS(n,t)$ is an upper finite \emph{symmetrically} based quasi-hereditary algebra.  However, this follows from observing that $\tau \left(1_{\bbeta} \right) = 1_{\bbeta}$ by \cref{C:Identity-Morphisms} and $\tau\left(X(\blambda, \bbeta)_{t} \right) = Y(\bbeta, \blambda)_{t}$ for all $\bbeta \in \Lambda(n)_{t}$ and $\blambda \in \Lambda^{+}(n)_{t}$ by definition. 
\end{proof}

As remarked in~\cite[Section 5]{BrundanStroppel}, thinking of $\SS(n,t)$ as a $\k$-linear category with this structure fits it into the theory of cellular categories~\cite{Westbury,EliasLauda}.

\subsection{Standard, Costandard, and Simple Modules for  \texorpdfstring{$\SS (n,t)$}{SS(n,t)}}\label{SS:Standard-Costandard-and-Simple-Modules-for-Snt}
The results of~\cite{BrundanStroppel} can be used to give $\Sntmodlfd$ the structure of a highest weight category and can be used to give a construction of the simple $\SS(n,t)$-modules.

For $\blambda\in\Lambda^{+}(n)_{t}$, let $\SS(n,t)_{\npreceq\blambda}$ be the quotient of $\SS(n,t)$ by the two-sided ideal, $\mathcal{I}(\npreceq \blambda)$, generated by the set
\[
\left\{ 1_{\bmu} \mid  \bmu\in\Lambda^{+}(n)_{t}, \bmu\npreceq\blambda\right\}.
\]
 Write $\bar{a}$ for the coset $a+\mathcal{I}(\npreceq \blambda)$. Let $\Delta(\blambda)$  denote the $\SS(n,t)$-module $\SS(n,t)_{\npreceq\blambda}\bar{1}_{\blambda}$ and let $\nabla(\blambda)$ denote the $\SS(n,t)$-module  $(\bar{1}_{\blambda}\SS(n,t)_{\npreceq\blambda})^{\circledast}$. Call $\Delta(\blambda)$ and $\nabla(\blambda)$ the standard and costandard modules, respectively, labelled by $\blambda$.

For each $\blambda \in \Lambda^{+}(n)_{t}$, $\Delta(\blambda)$ has a unique proper maximal submodule $R(\blambda)$.  Write 
\[
L(\blambda) = \Delta(\blambda)/R(\blambda) = \operatorname{head}(\Delta(\blambda))
\]
for the head of $\Delta(\blambda)$.  Write $\operatorname{socle}(N)$ for the socle of an $\SS(n,t)$-module $N$.

The following is~\cite[Theorem 5.9]{BrundanStroppel} applied to our setting.

\begin{theorem}\label{T:Snt-mod-is-a-HWC}  The category $\Sntmodlfd$ is a highest weight category with standard and costandard modules $\left\{ \Delta(\blambda)\right\}_{\blambda \in \Lambda^{+}(n)_{t}}$ and  $\left\{ \nabla(\blambda)\right\}_{\blambda \in \Lambda^{+}(n)_{t}}$, respectively.  The set 
\[
\left\{L(\blambda) = \operatorname{head}(\Delta(\blambda)) \cong \operatorname{socle}(\nabla(\blambda)) \mid \blambda \in \Lambda^{+}(n)_{t}  \right\}
\] is a complete, irredundant set of simple $\SS(n,t)$-modules.
\end{theorem}

\subsection{Semisimplicity}\label{SS:Semisimplicity-of-Snt}

\begin{theorem}\label{T:Snt-when-t-is-natural-is-not-semisimple}  Let $n \in \N $.  Then, $\Sntmod$ is a semisimple category if and only if $n=1$ or $n\geq 2$ and $t \not\in \Z_{\geq 0}$.
\end{theorem}
\begin{proof}  If $n=1$, then $\SS(1,t) = \k$ for all $t \in \k$ and semisimplicity is trivially true.  Therefore, we assume $n \geq 2$ for the remainder of the proof.

First, assume $t=d \in \Z_{\geq 0}$ and consider the surjective $\k$-algebra map $f_{n,d}: \SS(n,d) \to S(n,d)$.  By inflation through this map the simple modules for $S(n,d)$ define finite-dimensional simple modules for $\SS(n,d)$.

However, for all $\blambda \in \Lambda^{+}(n)_{t}$, $\Delta(\blambda)$ has a unique maximal proper submodule and therefore is indecomposable.  As noted in the discussion just before~\cite[Theorem 5.9]{BrundanStroppel}, $\Delta(\blambda)$ has a basis indexed by $Y(\blambda)_{t}$.  However, $Y(\blambda)_{t}$ is easily seen to be an infinite set.  For example, for all $k \geq 0$, $Y(\bdelta^{k}, \blambda )_{t}$ is nonempty when $\bdelta^{k}=(\lambda_{1}-k, \lambda_{2}+k, \lambda_{3}, \dotsc, \lambda_{n})$.  Namely, define $q = (q_{i,j})$ by
\begin{align*}
q_{1,1} &= \lambda_{1}-k,\\
q_{2,1} & = k, \\
q_{i,j} & = \delta_{i,j}\lambda_{i}, \text{ otherwise.}
\end{align*} This is easily seen to be an element of $Y(\bdelta^{k}, \blambda)_{t}$.
 
If $\Sntmod$ were a semisimple category when $t=d$, then $L(\blambda)=\Delta(\blambda)$ for all $\blambda \in \Lambda^{+}(n)_{t}$ and all simple modules would be infinite-dimensional.  But this contradicts the fact that finite-dimensional simple modules are known to exist.  

Now assume $t \in \k$ is not a nonnegative integer.  As discussed in~\cite[Section 2.2]{BrundanDavidson}, because $\SS(n,t)$ is locally finite-dimensional, it will be semisimple if and only if the Jacobson radical, 
\[
J(\SS(n,t)) := \bigcap_{\blambda \in \Lambda^{+}(n)_{t}} \operatorname{Ann}_{\SS(n,t)}\left(L(\blambda) \right),
\] is zero.  Since 
\[
J(\SS(n,t)) = \bigoplus_{\bbeta , \bdelta  \in \Lambda(n)_{t}} 1_{\bbeta }J(\SS(n,t))1_{\bdelta },
\] it will suffice to show each $1_{\bbeta }J(\SS(n,t))1_{\bdelta }$ is equal to zero.

Let $x \in 1_{\bbeta }J(\SS(n,t))1_{\bdelta }$ and write
\[
x=\sum_{q \in \MMat_{n}(\bbeta  , \bdelta )_{t}} c_{q}\xi_{q}.
\] But then, for any fixed $r^{*} \in \MMat_{n}(\bdelta  , \bbeta)_{t}$, 
\[
x\xi_{r^{*}}=\sum_{q \in \MMat_{n}(\bbeta  , \bdelta )_{t}} c_{q}\xi_{q}\xi_{r^{*}}
\] is an element of $1_{\bbeta }J(\SS(n,t)) 1_{\bbeta }$.

Because $\SS(n,t)$ is locally finite-dimensional over $\k$, it is also algebraic over $\k$. Therefore the Jacobson radical is a nil ideal (e.g., see~\cite[Lemma 3.12(iv)]{Passman} or~\cite[Corollary 4.19]{LamBook}).  However, if $x\xi_{r^{*}}$ is nilpotent, then \cite[Corollary 14.8, Theorem 7.11]{HarmanSnowden} imply that 
\[
\tr \left( x\xi_{r^{*}}\right) =0.
\]  On the other hand, by  \cref{L:basic-trace-facts}:
\begin{align*}
\tr \left( x\xi_{r^{*}}\right) &= \tr \left(\sum_{q \in \MMat_{n}(\bbeta  , \bdelta )_{t}} c_{q}\xi_{q}\xi_{r^{*}} \right)\\
& = \sum_{q \in \MMat_{n}(\bbeta  , \bdelta )_{t}} c_{q}\tr \left(\xi_{q}\xi_{r^{*}} \right)\\
& = c_{r}t(t-1)(t-2)\dotsb (t-m+1)z,
\end{align*} for some $z \in \Q_{>0}$ and nonnegative integer $m$.  Since $t\not\in \Z_{\geq 0}$, this expression equaling zero implies $c_{r}=0$.  Since this is true for all $r \in \MMat_{n}(\bbeta, \bdelta  )_{t}$, $x$ must equal zero.  Therefore, $J(\SS(n,t))=0$ and $\SS(n,t)$ is semisimple.
\end{proof}

\section{ \texorpdfstring{$\SS (n,t)$}{SS(n,t)} as a Quotient of a Modified Enveloping Algebra for  \texorpdfstring{$\gl_{n}(\k )$}{gln(k)}}\label{S:Enveloping-Algebra}

\subsection{The General Linear Lie Algebra}\label{SS:Lie-combinatorics} 

For $n \in \N  \cup \{\infty \}$ and a characteristic zero field $\k$, consider the general linear Lie algebra $\gl_{n}(\k )$ with entries from $\k$ (where all but finitely many entries are assumed to be zero in the case $n=\infty$).  Let $\fh=\fh_{n}$ denote the Cartan subalgebra consisting of diagonal matrices, $\fb^{+}=\fb_{n}^{+}$ the Borel subalgebra consisting of upper triangular matrices, and $\fb^{-}=\fb_{n}^{-}$ the opposite Borel subalgebra consisting of lower triangular matrices.

If we let $\varepsilon_{i}: \fh \to \k$ denote the linear functional which selects the $i$th diagonal entry when given an element of $\fh$, then $X(n)$ from \cref{SS:Compositions-and-Partitions} is the integral weight lattice for $\gl_{n}(\k)$; more generally, $X(n)_{t}$ from \cref{SSS:compositions-and-partitions-of-t} is a $t$-shifted integral weight lattice for any $t \in \k$.

For $1 \leq  i \neq j \leq  n$, let $\balpha_{i,j}= \varepsilon_{i}-\varepsilon_{j}$, and for $i \in [1,n-1]$ let $\balpha_{i}=\balpha_{i,i+1}$.  Observe that $\blambda + \balpha_{i,j}$ is an element of $X(n)_{t}$ for all $\blambda \in X(n)_{t}$ and all $\balpha_{i,j}$. The set $\Phi = \{\balpha_{i,j} \mid 1 \leq  i \neq j \leq n \}$ is the set of roots for $\gl_{n}(\k)$. The sets $\Phi^{\pm}=\{\pm \balpha_{i,j}\mid 1 \leq i < j \leq n \}$ and $\Delta=\{\balpha_{i}\mid 1 \leq i \leq n-1 \}$ are the choice of positive, negative, and simple roots corresponding to our choice of Borel subalgebra.  The matrix unit $E_{i,j} \in \gl_{n}(\k)$ is our fixed choice of root vector corresponding to $\balpha_{i,j} \in \Phi$, and $e_{i}:=E_{i,i+1}$ and $f_{i}:=E_{i+1,i}$ are our fixed choice of simple root vectors.

\subsection{A Locally Unital Enveloping Algebra}\label{SS:t-shifted-enveloping-algebra}  Fix $t$ in a commutative ring $R$.  Using the notation of the previous section, we define a locally unital $R$-algebra as follows.

\begin{definition}\label{D:Udot-Def}
  Let $\dotU(n,t)=\dotU_{R}(n,t)$ be the locally unital associative $R$-algebra with generators
\[
\{e_i^{\blambda},f_i^{\blambda} \mid \blambda\in X(n)_{t}, i \in [1,n-1]\}\cup\{1_{\blambda} \mid\blambda\in X(n)_{t}\}
\] 
and, for all $\blambda, \bmu \in X(n)_{t}$ and all $i,j \in [1,n-1]$, defining relations:
\begin{enumerate}
    \item[(U0)] $1_{\blambda }1_{\bmu }=\delta_{\blambda,\bmu}1_{\blambda}$;
    \item[(U1)] $e_{i}^{\blambda}1_{\bmu}=\delta_{\blambda,\bmu}e_{i}^{\blambda}$, $1_{\bmu}e_{i}^{\blambda}=\delta_{\blambda+\balpha_{i},\bmu}e_{i}^{\blambda}$;
    \item[(U2)] $f_{i}^{\blambda}1_{\bmu}=\delta_{\blambda+\balpha_{i},\bmu}f_{i}^{\blambda}$, $1_{\bmu}f_{i}^{\blambda}=\delta_{\blambda,\bmu}f_{i}^{\blambda}$;
    \item[(U3)]  $e^{\blambda +\balpha_{j}+\balpha_{i}}_ie^{\blambda +\balpha_{j}}_ie^{\blambda }_{j}-2e^{\blambda +\balpha_{i}+\balpha_{j}}_ie^{\blambda +\balpha_{i}}_je^{\blambda }_{i}+e^{\blambda +\balpha_{i}+\balpha_{i}}_je^{\blambda + \balpha_{i}}_ie^{\blambda }_{i}=0$,  if $|i-j|=1$;
    \item[(U4)]  $f^{\blambda}_if^{\blambda +\balpha_{i}}_if^{\blambda+\balpha_{i}+\balpha_{i}}_{j}-2f^{\blambda}_if^{\blambda +\balpha_{i}}_jf^{\blambda+\balpha_{i}+\balpha_{j} }_{i}+f^{\blambda}_jf^{\blambda +\balpha_{j}}_if^{\blambda +\balpha_{i}+\balpha_{j}}_{i}=0$,  if $|i-j|=1$;
    \item[(U5)]  $e^{\blambda+\balpha_{j}}_ie^{\blambda }_j-e^{\blambda + \balpha_{i}}_je^{\blambda}_i=0$,  if $|i-j| \neq 1$;
    \item[(U6)]  $f^{\blambda}_if^{\blambda +\balpha_{i} }_j-f^{\blambda }_jf^{\blambda+ \balpha_{j} }_i=0$,  if $|i-j| \neq 1$;
   \item[(U7)]  $e^{\blambda-\balpha_{j}}_if^{\blambda - \balpha_{j}}_j-f^{\blambda+\balpha_{i}-\balpha_{j}}_je^{\blambda }_i=\delta_{i,j}(\lambda_i-\lambda_{i+1})1_{\blambda}$.
\end{enumerate}  
\end{definition}  
\noindent Note that the only defining relation influenced by the choice of $t$ is (U7).  Also note that $\dotU_{R}(n,t)$ admits an anti-involution, $\tau$, given on generators by $\tau (1_{\blambda }) = 1_{\blambda }$, $\tau (e_{i}^{\blambda}) = f_{i}^{\blambda}$, and $\tau (f_{i}^{\blambda}) = e_{i}^{\blambda}$.

\begin{remark}\label{R:Udot-Remark-1} The algebra $\dotU_{R}(n,t)$ is a locally unital version of the enveloping algebra of the general linear Lie algebra, $U(\gl_{n}(\k))$, with  $t$-shifted integral weights.  In particular, $e_{i}^{\blambda}$ and $f_{i}^{\blambda}$ are locally unital versions of the simple root vectors $e_{i}$ and $f_{i}$, and the defining relations for $\dotU_{R}(n,t)$ are locally unital versions of the defining relations for $U(\gl_{n}(\k))$.  By definition, a $\dotU_{R}(n,t)$-module is a module, $M$, for the associative $R$-algebra $\dotU_{R}(n,t)$ which also satisfies $M = \bigoplus_{\bbeta \in X(n)_{t}} 1_{\bbeta}M$, where the direct sum is a direct sum of $R$-modules. 
\end{remark}

\begin{remark}\label{R:Udot-Remark} If $R=\k$ is a characteristic zero field, then the category of modules for the locally unital algebra $\dotU_{\k}(n,t)$ coincides with the category of weight modules for the enveloping algebra $U(\mathfrak{gl}_{n}(\k ))$ whose weights lie in $X(n)_{t}$.  

We briefly sketch the correspondence. If $M = \bigoplus_{\bgamma \in X(n)_{t}} M_{\bgamma}$ is a weight module for $U(\mathfrak{gl}_{n}(\k ))$, then define the action of $\dotU_{\k}(n,t)$ on $M$ by declaring for any $m \in M_{\bgamma}$ that $1_{\blambda }m = \delta_{\blambda , \bgamma}m$, $e_{i}^{\blambda}m = \delta_{\blambda , \bgamma} e_{i}m$, and $f_{i}^{\blambda}m = \delta_{\blambda+\balpha_{i} , \bgamma} f_{i}m$.  Conversely, given a $\dotU_{\k}(n,t)$-module, $M = \bigoplus_{\bgamma \in X(n)_{t}} 1_{\bgamma}M$,  declare $1_{\bgamma}M$ to be the $\bgamma$-weight space for the Cartan subalgebra,  and declare the action of $U(\mathfrak{gl}_{n}(\k ))$ to be given on an $m \in 1_{\bgamma}M$ by $e_{i}m = e_{i}^{\bgamma}m$ and $f_{i}m = f_{i}^{\bgamma-\balpha_{i}}m$.

We will freely go between these two notions in what follows.
\end{remark}

\subsection{Borel and Borel--Schur Subalgebras}\label{SSS:Borel-Subalgebras}

Given $t \in R$, let $\dotU (\fb^{+}, t)= \dotU_{R}(\fb^{+}, t)$ denote the locally unital $R$-algebra generated by $\left\{1_{\blambda} \mid \blambda \in X(n)_{t} \right\}$ and $\left\{e_{i}^{\blambda}\mid \blambda \in X(n)_{t}, i \in [1,n-1] \right\}$ subject to relations (U0), (U1), (U3), and (U5).  Likewise, let $\dotU (\fb^{-},t)=\dotU_{R}(\fb^{-},t)$ denote the locally unital $R$-algebra generated by $\left\{1_{\blambda} \mid \blambda \in X(n)_{t} \right\}$ and $\left\{f_{i}^{\blambda}\mid \blambda \in X(n)_{t}, i \in [1,n-1] \right\}$ subject to relations (U0), (U2), (U4), and (U6).

Because the relevant defining relations do not involve the parameter $t$, there is an isomorphism for every $t$, 
\[
\dotU_{R}(\fb^{\pm}):=\dotU_{R}(\fb^{\pm}, 0) \xrightarrow{\cong} \dotU_{R} (\fb^{\pm}, t),
\]
given by $1_{\blambda}\mapsto 1_{\blambda(t)}$, $e_{i}^{\blambda} \mapsto e_{i}^{\blambda (t)}$, and $f_{i}^{\blambda} \mapsto f_{i}^{\blambda (t)}$.
If $R$ contains $\Q$, then, 
\[
\dotU_{\Q}(\fb^{\pm}) \otimes_{\Q}R \cong \dotU_{R}(\fb^{\pm}).
\]

For $n \in \N \cup \{\infty \}$, $d \in \Z_{\geq 0}$, and $t \in R$, let $\SS^{\pm }_{R}(n,t)$ be the subalgebra of $\SS_{R}(n,t)$ spanned by $\xi_{q}$ for all $q \in \MMat^{\pm}_{n}(t)$.  Similarly, let $S^{\pm}_{R}(n,d)$ denote the subalgebras of $S_{R}(n,d)$ spanned by $\xi_{q}$ for all $q \in \Mat^{\pm}(n,d)$.  The fact that these are subalgebras follows from \cref{T:product-formula,T:classical-product} (see the proof of \cref{P:Borel-Subalgebra-maps} for a similar argument).  

\begin{proposition}\label{P:Borel-Subalgebra-maps} Let $n \in \N \cup \{\infty \}$, $d \in \Z_{\geq 0}$, and $t \in R$.  There are homomorphisms of locally unital algebras, 
\begin{align*}
\iota_{n,t}^{U^{\pm}}:\dotU_{R}(\fb_{n}^{\pm}) \to \dotU_{R}(n,t), \\
\iota_{n,d,t}^{S^{\pm}}:S^{\pm}_{R}(n,d) \to \SS^{\pm}_{R}(n,t).
\end{align*} 
\end{proposition}

\begin{proof}  We only do the $+$ cases since the $-$ cases are essentially identical. The map $\iota_{n,t}^{U^{+}}$ is the obvious one given on generators by $1_{\blambda} \mapsto 1_{\blambda (t)}$ and $e_{i}^{\blambda}\mapsto e_{i}^{\blambda(t)}$ for all $\blambda \in X(n)_{0}$ and $i \in [1,n-1]$.  Since the defining relations of $\dotU_{R}(\fb_{n}^{\pm})$ hold in $\dotU_{R}(n,t)$, it is immediate that this determines a well-defined homomorphism.

The map $\iota_{n,d,t}^{S^{+}}$ is given on basis elements by $\xi_{q} \mapsto \xi_{q(t)}$ for $q \in \Mat^{+}(n,d)$.  This evidently defines an embedding of $R$-modules.  It remains to verify the map is multiplicative.

First, we claim that for all $\bdelta , \bgamma \in \Lambda(n,d)$ the map $q \mapsto q(t)$ is a bijection between $\Mat^{+}(n, d,  \bdelta  , \bgamma  )$ and $\MMat^{+}_{n}(\bdelta(t)  , \bgamma(t)  )_{t}$.  That this is an injective function between the stated sets is immediate.  It remains to see that it is surjective.  To do so, it suffices to verify $r(d)_{1,1} \geq 0$ for all upper triangular $r \in \MMat_{n}(\bdelta(t)  , \bgamma(t)  )_{t}$.  However, 
\[
r(d)_{1,1}= d - \sum_{(i,j) \in [1,n]^{2}\backslash \{(1,1) \}} r(d)_{i,j} = d - \sum_{(i,j) \in [1,n]^{2}\backslash \{(1,1) \}} r_{i,j} = d - \sum_{j \in [2,n]} \delta_{j} = \delta_{1} \geq 0,
\] because $r$ is upper triangular and $\bdelta = (\delta_{1}, \dotsc , \delta_{n}) \in \Lambda(n,d)$.

 We next claim that if $q,r \in \Mat^{+}(n,d)$ and $C_{q(t),r(t)}^{s}(t) \neq 0$ for some $s \in \MMat_{n}(\bdelta (t), \bbeta (t))_{t}$, then $s$ is upper triangular. Say $A = (A_{i,j,k}) \in \AA_{n}(q(t), r(t), s)$ is an array used to calculate $C_{q(t),r(t)}^{s}(t)$, then $\sum_{k \in [1,n]} A_{i,j,k} = q_{i,j}=0$ whenever $i > j$ and $\sum_{i \in [1,n]} A_{i,j,k} = r_{j,k}=0$ whenever $j > k$.  This implies $A_{i,j,k}=0$ whenever $i > j$ or $j > k$.  However, if $i > k $, then this implies the nonnegative integer $s_{i,k}$ satisfies
\[
s_{i,k} = \sum_{j \in [1,n]} A_{i,j,k} \leq \sum_{j \in [1,i-1]} A_{i,j,k} + \sum_{j \in [k+1,n]} A_{i,j,k} = 0.
\]  That is, $s$ is upper triangular.  The same argument shows that if $q$ and $r$ are upper triangular and $c_{q,r}^{s} \neq 0$ for some $s \in \Mat (n, d, \bdelta, \bbeta)$, then $s$ is upper triangular.

Let $r \in \Mat^{+}(n, d, \bgamma , \bbeta  )$ and $q \in \Mat^{+}(n, d, \bdelta  , \bgamma  )$.  Then, 
\begin{equation}\label{E:product-equation}
 \xi_{q(t)}\xi_{r(t)} = \sum_{\tilde{s} \in \MMat^{+}_{n}(\bdelta (t), \bbeta (t))_{t}} C_{q(t),r(t)}^{\tilde{s}}(t) \xi_{\tilde{s}}=\sum_{s \in \Mat^{+}(n, d, \bdelta, \bbeta)} C_{q(t),r(t)}^{s(t)}(t) \xi_{s(t)},
\end{equation} where the last equality follows from the preceding discussion.
The map $\iota^{S^{+}}_{n,d,t}$ will be multiplicative as long as
\[
C_{q(t),r(t)}^{s(t)}(t) = c_{q,r}^{s}
\]
for all upper triangular $q,r,s \in \Mat^{+}(n,d)$.

First, we claim that the function $\A(n, \qrs ) \to \AA_{n}(q(t), r(t), s(t))_{t}$ given by $A\mapsto A(t)$ is a bijection.  Like before, the map is clearly injective.  Surjectivity follows if we verify that $B(d)_{1,1,1} \geq 0$ for all $B\in \AA_{n}(q(t), r(t), s(t))_{t}$.  However, since $q(t)$ is upper triangular, $B_{i,1,1}(d)=B_{i,1,1}=0$ for $i \in [2,n]$.  As a consequence,
\[
B(d)_{1,1,1}=  \sum_{i \in [1,n]} B_{i,1,1}(d) = r_{1,1} \geq 0, 
\] where the last inequality holds because $r$ was assumed to be an element of $\Mat(n,d)$.

We next claim
\[
\VV_{A}(t) = v_{A(d)}
\] for all $A \in \AA_{n}(q(t), r(t), s(t))_{t}$. In comparing the two, we see there is only one term in the combinatorial formulas for $\VV_{A}(t)$ and $v_{A(d)}$ that might differ due to the shift in the $(1,1,1)$-entry.  However, $r(t)$ is upper triangular and this implies
\[
\sum_{i \in [1,n]} A_{i,j,k} = r(t)_{j,k}=0
\] whenever $j > k$.  Thus, $A_{i,j,k}=A(t)_{i,j,k}=0$ for $j > k$.  In particular, $A_{1,j,1}=A(t)_{1,j,1}=0$ for $j \in [2,n]$.  Applying this observation to the combinatorial formulas for $\VV_{A}(t)$ and $v_{A(d)}$ shows that the one term where they could differ is actually equal to $1$ in both of them.  Hence the two formulas agree and $\VV_{A}(t) = v_{A(d)}$. 

From this it follows that the structure constants $C_{q(t),r(t)}^{s(t)}$ and $c_{q,r}^{s}$ are equal, as desired.
\end{proof}

\begin{remark}\label{R:iota-properties}

The family of maps $\left\{\iota_{n,d,t}^{S^{\pm}} \right\}_{d \in \N}$ are asymptotically surjective in the sense that, given $y \in \SS^{\pm}(n,t)$, for all sufficiently large $d$ there exists $x \in S^{\pm}(n,d)$ such that $\iota^{S^{\pm}}_{n,d,t}(x) = y$.  This follows from the fact that any finite subset $\MMat_{n}(t)$ is in the image of the map $\Mat (n,d) \to \MMat_{n}(t)$ given by $q \mapsto q(t)$, once $d$ is sufficiently large.
\end{remark}

\subsection{Classical Schur--Weyl Duality}\label{SS:classical-Schur-Weyl-duality}

Since the $U(\gl_{n}(\Q))$-module $V_{n}^{\otimes d}$ has weights which lie in $X(n)_{d}$, we may view $V_{n}^{\otimes  d}$ as a $\dotU_{\Q}(n,d)$-module.  From this and classical Schur--Weyl duality, it follows that for every $n,d \in \N$ there is a surjective map, 
\begin{equation*}
\pi_{n,d}: \dotU_{\Q}(n,d) \to S_{\Q }(n,d).
\end{equation*}  The restriction of these maps define surjective maps 
\begin{equation*}
\pi_{n,d}: \dotU_{\Q}(\fb^{\pm}) \to S_{\Q }^{\pm}(n,d),
\end{equation*}
which we call by the same name.  That these maps remain surjective can be deduced from the explicit description of these maps obtained by specializing the $q$ to $1$ in~\cite[Proposition 2.3]{GreenRM}.

If $R$ contains $\Q$, then after base change we have a surjective map which we continue to call by the same name,
\begin{equation}\label{E:Borel-pi-maps}
\pi_{n,d}: \dotU_{R}(\fb^{\pm}) \to S_{R}^{\pm}(n,d).
\end{equation}

\subsection{Summary of Algebra Maps}\label{SS:Maps-Among-Enveloping-and-Schur-Algebras}

Let $d \in \N$ and let $\ev_{d}: \Q [T] \to \Q$ be evaluation at $T=d$.  There is a corresponding map of rings, 
\begin{equation}\label{E:ev-family-of-maps}
\ev_{d}: \dotU_{\Q [T]}(n,T) \to \dotU_{\Q}(n,d),
\end{equation}
which we call by the same name.

For all $d \in \N$, the locally unital algebra $\dotU_{\Q}(n,d)$ has a basis given by PBW elements which are certain ordered products of root vectors (e.g., this can be seen by specializing the $q$ in~\cite[Theorem 3.1]{WangW-PBW} to $1$).  Standard inductive arguments on monomial length show that $\dotU_{\Q [T]}(n,T)$ is spanned as a $\Q [T]$-algebra by the same ordered products of root vectors, which we call PBW elements.  Since $\ev_{d}$ takes generators to generators, it also takes root vectors to root vectors, products of root vectors to products of root vectors and, hence these PBW elements to PBW basis elements in $\dotU_{\Q}(n,d)$.

Say $\sum_{i}a_{i}(T)X_{i}=0$ is a linear combination of distinct PBW elements in $\dotU_{\Q [T]}(n,T)$.  Then $\ev_{d}\left(\sum_{i}a_{i}(T)X_{i} \right)=\sum_{i}a_{i}(d)X_{i}=0$ for all $d \gg 0$.  However, the PBW elements  in $\dotU_{\Q}(n,d)$ are a basis and, hence, $a_{i}(d)=0$ for all $d \gg 0$.  Therefore $a_{i}(T)=0$ for all $i$ and the PBW elements of $\dotU_{\Q [T]}(n,T)$ form a basis.

A similar argument shows that the family of maps $\left\{\ev_{d} \right\}_{d \in \N }$ given in \cref{E:ev-family-of-maps} is asymptotically injective in the sense that $\ev_{d}(x) = 0$ for all $d \gg 0$ implies $x=0$.  Because these maps send a basis to a basis, extension of scalars along $\ev_{d}: \Q  [T] \to \Q $ defines an isomorphism
\[
\dotU_{\Q [T]}(n,T) \otimes_{\Q [T]} \Q \cong \dotU_{\Q}(n,d)
\]   for all $d \in \Z_{\geq 0}$.

The various maps now in play are summarized in the following diagram of ring homomorphisms. The map $u_{n}$ will be introduced shortly and will make the diagram commute.
\begin{equation}\label{E:UnT-SnT-commutative-square}
\begin{tikzcd}
	\dotU_{\Q [T]}(\fb^{\pm}_{n} ) & & S^{\pm}_{\Q [T]} (n,d)  \\
	  & & \\
	\dotU_{\Q [T]}(n,T) & & \SS_{\Q [T]} (n,T)  \\
	  & & \\
	\dotU_{\Q}(n,d) & & S_{\Q}(n,d)
	\arrow["\pi_{n,d}", two heads = right, from=1-1, to=1-3]
	\arrow["\iota^{U^{\pm}}_{n,T}", from=1-1, to=3-1]
	\arrow["\iota^{S^{\pm}}_{n,d,T}", from=1-3, to=3-3]
	\arrow["\ev_{d}", from=3-1, to=5-1]
	\arrow["{u_n}", dashed, from=3-1, to=3-3]
	\arrow["{F^{\Q}_{n,d}}", from=3-3, to=5-3]
	\arrow["\pi_{n, d}"', two heads = right, from=5-1, to=5-3]
\end{tikzcd}
\end{equation}

\subsection{Map from the Locally Unital Enveloping Algebra to the Interpolating Schur Algebra}\label{SS:map-from-enveloping-to-schur}  For $\blambda \in \Lambda(n)_{t}$, and $i \in [1,n-1]$, let $q_{i}^{\blambda,+} \in \MMat_{n} (\blambda + \balpha_{i}, \blambda )_{t}$ and $q_{i}^{\blambda,-} \in \MMat_{n} (\blambda , \blambda + \balpha_{i} )_{t}$ be given by
\[
q_{i}^{\blambda,+}=\left(\begin{tikzcd}[cramped,sep=tiny]
	{\lambda_1} &&&&& \\
	& \ddots &&&& \\
	&& {\lambda_{i}} & 1 && \\
	&&& {\lambda_{i+1}-1} && \\
	&&&& \ddots \\ &
	&&&&& {\lambda_{n}}
\end{tikzcd}\right) \;\; \text{ and } \;\;
q_{i}^{\blambda,-}=\left(\begin{tikzcd}[cramped,sep=tiny]
	{\lambda_1} &&&&& \\
	& \ddots &&&& \\
	&& {\lambda_{i}} &&& \\
	&& 1 & {\lambda_{i+1}-1} && \\
	&&&& \ddots \\ &
	&&&&& {\lambda_{n}}
\end{tikzcd}\right).
\]

\begin{theorem}\label{T:map-from-dotU-to-Snt}  Let $R= \Q [T]$ and $t = T$.  Then, there is a map of $\Q [T]$-algebras
\[
u_n:\dotU_{\Q [T]}(n,T)\to \SS_{\Q [T]}(n,T)
\]
given on generators by

\begin{align*}
    e_i^{\blambda}&\mapsto \begin{cases}
        \xi_{q_{i}^{\blambda, +}}, & \text{if $\blambda, \blambda + \balpha_{i} \in\Lambda(n)_{T}$};\\
        0, & \text{else};
    \end{cases} \\
     f_i^{\blambda}&\mapsto \begin{cases}
        \xi_{q_i^{\blambda,-}}, & \text{ if $\blambda, \blambda+ \balpha_{i} \in\Lambda(n)_{T}$};\\
        0, & \text{else};
    \end{cases}\\
     1_{\blambda}&\mapsto \begin{cases}
        1_{\blambda}, & \text{if $\blambda\in\Lambda(n)_{T}$};\\
        0, & \text{else}.
    \end{cases}
\end{align*}
\end{theorem}

\begin{proof}  First, observe that a calculation on the basis $\{v_{\bi}\}_{\bi \in I(n,d)}$  verifies that $F^{\Q}_{n,d}\left( \xi_{q_{i}^{\blambda, +}}\right)$ and  $F^{\Q}_{n,d}\left( \xi_{q_{i}^{\blambda, -}}\right)$ act on $V_{n}^{\otimes d}$ in the same way as $\pi_{n,d}\left(\ev_{d}\left(e_{i}^{\blambda}\right)\right)$ and $\pi_{n,d}\left(\ev_{d}\left(f_{i}^{\blambda}\right)\right)$, respectively.  Therefore, the bottom square of \cref{E:UnT-SnT-commutative-square} commutes, assuming the given rules for $u_{n}$ on generators define a homomorphism.  Likewise, a check on generators demonstrates that the top square of \cref{E:UnT-SnT-commutative-square} commutes, assuming $u_{n}$ exists.

Our task, then, is to verify that the defining relations for $\dotU_{\Q [T]}(n,T)$ go to zero in $\SS_{\Q [T]}(n,T)$ under $u_{n}$.  This can be done directly using \cref{T:product-formula} and elbow grease.   We opt for an easier, albeit less direct, approach.

Let $\mathtt{r} \in 1_{\bdelta}\SS_{\Q [T]}(n,T) 1_{\bgamma}$ be one of the defining relations of $\dotU_{\Q [T]}(n,T)$ expressed as an element of $\SS_{\Q [T]}(n,T)$ by applying the rule for $u_{n}$ to the generators.   For example, $\mathtt{r}$ could equal 
\[
\left( \xi_{q_{i}^{\bgamma+\balpha_{j},+}}\right)\left( \xi_{q_{j}^{\bgamma,+}}\right)-\left(\xi_{q_{j}^{\bgamma+\balpha_{i},+}}\right)\left(\xi_{q_{i}^{\bgamma,+}}\right)
\] in $1_{\bgamma+\balpha_{i}+\balpha_{j}}\SS_{\Q [T]}(n,T) 1_{\bgamma}$.  Verifying relation (U5) amounts to showing this element equals zero in $\SS_{\Q [T]}(n,T)$ for all $\bgamma \in \Lambda(n)_{T}$.  As discussed in \cref{R:integral-evaluation-plus-quotient}, the family of maps $\{F_{n,d}^{\Q} \}_{d \in \N}$ is asymptotically injective.  Therefore, showing $\mathtt{r}$ equals zero is equivalent to showing that  $F_{n,d}^{\Q}(\mathtt{r})=0$ for all large $d$.  It follows from the discussion in the first paragraph that
\[
F_{n,d}^{\Q}(\mathtt{r}) = \pi_{n,d}\left( \ev_{d}\left( e_{i}^{\bgamma+\balpha_{j}}e_{j}^{\bgamma}-e_{j}^{\bgamma+\balpha_{i}}e_{i}^{\bgamma}\right) \right)
\] as maps on $V_{n}^{\otimes d}$. However, this map is zero since $e_{i}^{\bgamma+\balpha_{j}}e_{j}^{\bgamma}-e_{j}^{\bgamma+\balpha_{i}}e_{i}^{\bgamma}$ equals zero in $\dotU_{\Q [T]}(n,T)$ by relation~(U5).   Therefore, $\mathtt{r}=0$ in $\SS_{\Q [T]}(n,T)$, as desired.  The same argument applies equally well to the other defining relations.
\end{proof}

\begin{theorem}\label{T:un-is-surjective} Let $R= \Q [T]$ and $t = T$.  The map of $\Q [T]$-algebras
\[
u_n:\dotU_{\Q [T]}(n,T)\to \SS_{\Q [T]}(n,T)
\] is surjective.
\end{theorem}
\begin{proof}  Let $q \in \MMat^{+}_{n}(T)$.  As stated in \cref{R:iota-properties}, $\iota^{S^{+}}_{n,d,T}: S^{+}_{\Q [T]}(n,d) \to \SS^{+}_{\Q [T]}(n,T)$ is asymptotically surjective.  That is, we may choose $d$ large enough to ensure that $q(d)$ is an element of $\Mat^{+}(n,d)$ and also that $\iota^{S^{+}}_{n,d,T}(\xi_{q(d)}) = \xi_{q}$.  The map $\pi_{n,d}: \dotU_{\Q[T]}(\fb^{+}) \to S^{+}_{\Q [T]}(n,d)$ given in \cref{E:Borel-pi-maps} is surjective.  Therefore, we can choose $X_{q(d)} \in \dotU_{\Q[T]}(\fb^{+})$ such that $\pi_{n,d}(X_{q(d)}) = \xi_{q(d)}$ and, hence, 
\[
u_{n}\left( \iota^{U^{+}}_{n,T}\left( X_{q(d)}\right)\right)=\iota^{S^{+}}_{n,d,T}\left( \pi_{n,d}\left( X_{q(d)}\right)\right) = \xi_{q}.
\] Therefore $\SS^{+}_{\Q [T] }(n,T)$ lies in the image of $u_{n}$.  A similar argument shows $\SS^{-}_{\Q [T] }(n,T)$ lies in the image of $u_{n}$.  Finally, since the codeterminant basis $\CoDetBasis  (n)_{T}$ is a subset of $\SS^{-}_{\Q [T] }(n,T)\SS^{+}_{\Q [T] }(n,T)$, we conclude that $u_{n}$ is surjective.
\end{proof}

Notice that the map $u_{n}$ satisfies $u_{n}\circ\tau =\tau \circ u_{n}$. For a characteristic zero field $\k$ and $t \in \k$, the previous result, along with base change along the evaluation map $\ev_{d}: \Q [T] \to \k$ at $T=t$, yields the following.
\begin{corollary}\label{C:un-is-surjective}  There is a surjective map of locally unital $\k$-algebras
\begin{equation}\label{E:surjective-psi-map}
\psi = \psi_{n,t}:\dotU_{\k}(n,t)\twoheadrightarrow \SS_{\k}(n,t)
\end{equation}
for all $n \in \N \cup \{\infty \}$ and $t \in \k$.
\end{corollary}

\begin{remark}\label{R:presentation-of-Snt}  Like in \cite{DotyGiaquinto}, the above result should provide a presentation for $\SS_{\k}(n,t)$ by generators and relations.  We expect the kernel of the map given in \cref{C:un-is-surjective} will be the two-sided ideal generated by the set $\left\{1_{\blambda} \mid \blambda \in X(n)_{t} \backslash \Lambda(n)_{t} \right\}$.  See~\cite{Ryba} for related results.  Since we do not need this description, we leave it as a question.
\end{remark}

\subsection{Parabolic Category  \texorpdfstring{$\mathcal{O}$}{O}}\label{SS:Parabolic-Category-O}

Throughout this section $\k$ is a characteristic zero field and $t \in \k$. Let $\Phi_{\fl} = \left\{\balpha_{i,j} \mid 2 \leq  i \neq j \leq  n \right\}$ and let $\fl \cong \gl_{1}(\k ) \oplus \gl_{n-1}(\k)$ be the Levi subalgebra of $\gl_{n}(\k)$ corresponding to the subroot system $\Phi_{\fl}$.  Let $\fp^{\pm}$ be the parabolic subalgebra corresponding to $\Phi_{\fp^{\pm}} = \Phi_{\fl} \cup \Phi^{\pm}$.  Likewise, let $\fu^{\pm}$ be the nilpotent subalgebra corresponding to $\Phi_{\fp^{\pm}} \backslash \Phi_{\fl}$. 

Corresponding to these choices, parabolic category $\mathcal{O}^{\fp}$ is defined to be the full subcategory of $U(\gl_{n}(\k))$-modules $M$ that satisfy the following conditions:
\begin{itemize}
\item [(P1)] $M$ is finitely generated as a $U(\gl_{n}(\k ))$-module;
\item [(P2)] $M$ is finitely semisimple when restricted to $U(\fl)$ (i.e., $M$ is isomorphic to a direct sum of finite-dimensional simple $U(\fl )$-modules);
\item [(P3)] $M$ is locally $U(\fu^{+})$-finite (i.e., every element of $M$ generates a finite-dimensional $U(\fu^{+})$-module).
\end{itemize}

For a fixed $t \in \k$, let $\mathcal{P}_{t}$ be the full subcategory of finitely generated $U(\gl_{n}(\k ))$-modules which have a weight space decomposition with weights lying in $\Lambda(n)_{t}$.

\begin{lemma}\label{L:Pt-is-a-subcategory-of-Op}
The category $\mathcal{P}_{t}$ is a full subcategory of $\mathcal{O}^{\fp}$.
\end{lemma}

\begin{proof} 
Condition (P1) is obviously satisfied.

Let $N$ be a module in $\mathcal{P}_{t}$.  If $N$ is restricted to the subalgebra of $\fl$ isomorphic to $\gl_{n-1}(\k)$, then the weights of $N$  lie in $\Lambda(n-1)$.  For $p \in \Z_{\geq 0}$, let 
\[
N_{p} = \bigoplus_{\substack{\bmu \in \Lambda(n)_{t} \\ \sum_{i \in [2,n]} \mu_{i} =p }} N_{\bmu }.
\]  The presentation given in \cite[Theorem 2.1]{DotyGiaquinto} demonstrates that $N_{p}$ is a $S(n-1,p)$-module and, hence, is a direct sum of finite-dimensional simple $S(n-1,p)$-modules.  Therefore, via the canonical surjective map $U(\gl_{n-1}(\k )) \to S(n-1,p)$, $N_{p}$ is a direct sum of finite-dimensional simple $U(\gl_{n-1}(\k))$-modules.  However, since $N_{p}$ is spanned by weight vectors, it follows that $N_{p}$ is a finitely semisimple $U(\fl )$-module.  Therefore 
\[
N = \bigoplus_{p \geq 0} N_{p}
\] is a finitely semisimple $U(\fl)$-module.  Therefore, condition (P2) holds.

Finally, to show $N$ satisfies condition (P3) it suffices to show that each weight vector $x \in  N$ generates a finite-dimensional $U(\fu^{+})$-module. Since all weights of $N$ lie in $\Lambda(n)_{t}$, weight considerations show that only finitely many PBW basis vectors in $U(\fu^{+})$ can be applied to a weight vector of $N$ and still be nonzero.  Therefore, $U(\fu^{+})x$ is finite-dimensional, as desired.   
\end{proof}
 
Via the surjective map $\psi$ given in \cref{E:surjective-psi-map} and the identification of $\dotU_{\k}(n,t)$-modules and $U(\gl_{n}(\k))$-modules with a weight space decomposition as described in \cref{R:Udot-Remark}, there is a fully faithful inflation functor
\[
\Psi=\Psi_{n,t}:\Sntmod \to U(\gl_{n}(\k))\text{-mod}.
\]  

\begin{lemma}\label{L:Snt-modules-to-Category-O} The functor $\Psi$ restricts to define a fully faithful functor 
\[
\Psi: \Sntmodfg \to \mathcal{P}_{t}.
\]
\end{lemma}

\begin{proof}  It is immediate that the inflation of any finitely generated $\SS(n,t)$-module is finitely generated, has a weight space decomposition, and the weights lie in $\Lambda(n)_{t}$.  Therefore, the image of $\Psi$ lies in $\mathcal{P}_{t}$.  The fact that $\Psi$ is fully faithful follows from the fact that $\psi$ is surjective.
\end{proof}

\begin{remark}\label{R:conjectural-category-equivalence}  If the presentation of $\SS(n,t)$ given in \cref{R:presentation-of-Snt} is valid, then the category $\Sntmodfg$ is isomorphic to $\mathcal{P}_{t}$ via the functor $\Psi$. 
\end{remark}

We fix $\fh_{0}=\fh$ and $\fb_{0}=\fb \cap \fl$ as our choice of Cartan and Borel subalgebras of $\fl$, respectively.  Given $\blambda \in X(n)_{t}$, let $L_{\fl}(\blambda )$ be the simple $\fl$-module of highest weight $\blambda$ with respect to these subalgebras of $\fl$.  Let
\[
V(\blambda ) = U(\gl_{n}(\k )) \otimes_{U(\fp^{+})} L_{\fl}(\blambda )
\] be the generalized (or parabolic) Verma module for $\gl_{n}(\k )$ with highest weight $\blambda$.  Let $\widehat{L}(\blambda )$ be the simple head of $V(\blambda )$.

\begin{proposition}\label{L:standard-is-verma}  Let $t \in \k$ and let $\blambda \in \Lambda^{+}(n)_{t}$, then 
\[
\Psi_{n,t}(\Delta(\blambda )) \cong V(\blambda ).
\]
\end{proposition}
\begin{proof}   By construction, 
\[
\Delta(\blambda ) = \bigoplus_{\substack{ \bmu \in X(n)_{t}\\ \blambda  \not\preceq \bmu }} \Delta(\blambda )_{\bmu}.
\]  The condition $ \blambda  \not\preceq \bmu $ is precisely the requirement that $\bmu$ is not greater than $\blambda$ in the dominance order on weights of $\gl_{n}(\k)$.  Thus, by construction, $\Delta(\blambda )$ is generated by $\bar{1}_{\blambda}$, a highest weight vector of weight $\blambda$ which generates an $\fl$-composition factor isomorphic to $L_{\fl}(\blambda )$.  By the universal property of $V(\blambda )$, there is a surjective homomorphism 
\begin{equation}\label{E:Verma-to-Standard}
V(\blambda ) \twoheadrightarrow \Delta(\blambda ).
\end{equation}

It is well known (e.g., by~\cite[Theorem 9.12]{Humphreys}) that  $V(\blambda )$ is simple whenever $t$ is not an integer and, hence, the map is an isomorphism in this case.  In particular, 
\[
\dim_{\k} \left(  V(\blambda )_{\bmu}\right) = \dim_{\k} \left( \Delta(\blambda )_{\bmu} \right)
\] for all $\bmu \in X(n)_{t}$.  If we show these dimensions are always finite and are independent of $t$, then the map given in \cref{E:Verma-to-Standard} is an isomorphism for all $t$.  However, $\dim_{\k} \left(  V(\blambda )_{\bmu}\right)$ can be computed using the parabolic version of Kostant's partition function.  On the other hand, $\dim_{\k} \left( \Delta(\blambda )_{\bmu} \right) = | Y(\bmu , \blambda )_{t}|$.  Therefore, both dimensions are finite numbers that can be computed combinatorially without reference to $t$.
\end{proof}

The functor $\Psi$ sends the head of $\Delta(\blambda)$ to the head of $V(\blambda)$. With this in mind, the following result is immediate.
\begin{corollary}\label{C:Decomposition-Multiplicities-and-Category-O}
For all $\blambda, \bmu$ in $\Lambda^{+}(n)_{t}$, $\Psi_{n,t}\left(L(\blambda) \right) = \widehat{L}(\blambda )$ and the composition multiplicities for $\SS(n,t)$ satisfy:
\[
 \left[\Delta (\blambda) : L(\bmu )  \right] = \left[V(\blambda ) : \widehat{L}(\bmu ) \right].
\]
\end{corollary}

\makeatletter
\renewcommand*{\@biblabel}[1]{\hfill#1.}
\makeatother

\bibliographystyle{alpha}
\bibliography{interpolating-schur-algebras}

\end{document}